\documentclass[10pt]{amsart}
\usepackage[all]{xy}
\usepackage{amssymb}
\usepackage{enumitem}
\usepackage{amsthm}
\usepackage{hyperref}
\hypersetup{colorlinks=true,linkcolor=blue,citecolor=magenta}
\usepackage{amsmath}
\usepackage{amscd,enumitem}
\usepackage{verbatim}
\usepackage{eurosym}
\usepackage{float}
\usepackage{color}
\usepackage{dcolumn}
\usepackage[mathscr]{eucal}
\usepackage[all]{xy}
\usepackage{hyperref}
\usepackage{bbm}
\usepackage[textheight=8.9in, textwidth=6.85in]{geometry}
\DeclareFontFamily{U}{wncy}{}
    \DeclareFontShape{U}{wncy}{m}{n}{<->wncyr10}{}
    \DeclareSymbolFont{mcy}{U}{wncy}{m}{n}
    \DeclareMathSymbol{\Sh}{\mathord}{mcy}{"58} 
\newtheorem*{thm*}{Theorem}
\newtheorem*{conj*}{Conjecture}

\newtheorem*{remark}{Remark}

\newtheorem{theorem}{Theorem}[section]

\newtheorem{lemma}[theorem]{Lemma}

\newtheorem*{definition}{Definition}

\newtheorem*{example}{Example}
\newtheorem{corollary}[theorem]{Corollary}

\newcommand{\Z}{\mathbb{Z}}
\newcommand{\Q}{\mathbb{Q}}
\newcommand{\F}{\mathbb{F}}
\newcommand{\R}{\mathbb{R}}
\newcommand{\NN}{\mathcal{N}}

\newcommand{\tor}{\mathrm{tor}}
\newcommand{\height}{\mathrm{ht}}
\newcommand{\Tam}{\mathrm{Tam}}

\newcommand{\leg}[2]{\genfrac{(}{)}{}{}{#1}{#2}}

\numberwithin{equation}{section}

\newcommand{\tp}{c}

\begin{document}
\title[Tamagawa products of elliptic curves]{Tamagawa products of elliptic curves over $\Q$}
\author{Michael Griffin, Ken Ono and Wei-Lun Tsai}
\address{Department of Mathematics, 275 TMCB, Brigham Young University, Provo, UT 84602}
\email{mjgriffin@math.byu.edu}
\address{Department of Mathematics, University of Virginia, Charlottesville, VA 22904}
\email{ken.ono691@virginia.edu}
\email{tsaiwlun@gmail.com}

\begin{abstract}   
We explicitly construct the Dirichlet series
$$L_{\Tam}(s):=\sum_{m=1}^{\infty}\frac{P_{\Tam}(m)}{m^s},$$
where $P_{\Tam}(m)$ is the proportion of elliptic curves $E/\Q$ in short Weierstrass form  with Tamagawa product $m.$ 
Although there are no $E/\Q$ with everywhere good reduction, we prove that
the proportion with trivial Tamagawa product is $P_{\Tam}(1)={\color{black}0.5053\dots.}$
As a corollary, we find that $L_{\Tam}(-1)={\color{black}1.8193\dots}$ is the {\it average} Tamagawa product for  elliptic curves over $\Q.$  
We give an application of these results to canonical and Weil heights.
\end{abstract}

\maketitle
\section{Introduction and statement of results}\label{Intro}
 
It is well known that there are no elliptic curves $E/\Q$ with everywhere good reduction 
(for example, see Ch. VII-XIII of \cite{SilvermanFirstBook}).  In spite of this fact, there
are many $E/\Q,$ such as 
$$
E/\Q: \ \ y^2= x^3 -3x-4,
$$
(i.e. 5184.m1 in \cite{LMFDB})
with the weaker property that $[E(\mathbb{Q}_p):E_0(\mathbb{Q}_p)]=1$ for all primes  $p$, 
where $E_0(\mathbb{Q}_p)$ is the open subgroup of $E(\mathbb{Q}_p)$ consisting of nonsingular points.
These {\it Tamagawa trivial curves} 
have
\begin{equation}\label{TamagawaTrivial}
\Tam(E):=\prod_{p\ prime} c_p=1,
\end{equation}
where $c_p:= [E(\mathbb{Q}_p):E_0(\mathbb{Q}_p)]$ is the usual Tamagawa number at $p$.

Tamagawa trivial curves enjoy properties that motivate this note.
For example, if $E/\Q$ is  a Tamagawa trivial curve for which $E(\Q)$ has rank $r,$ then the Birch and Swinnerton-Dyer Conjecture predicts that
$$
\frac{L^{(r)}(E,1)}{r!}=\frac{|\Sh(E)| \cdot \Omega_E R_E}{|E_{\tor}(\Q)|^2}.
$$
Here $L(E,s)$ is the Hasse-Weil $L$-function for $E/\Q$,
$\Sh(E)$ is the Shafarevich-Tate group, $\Omega_E$ is the real period, $R_E$ is the regulator, and $E_{\tor}(\Q)$ is the $\Q$-rational torsion subgroup. Tamagawa trivial elliptic curves also play a prominent role  in the work of Balakrishnan, Kedlaya, and Kim \cite{BKK} that offers the first explicit positive genus examples of nonabelian Chabauty: the case of quadratic Chabauty for determining integral points on rank 1 elliptic curves
 \cite{BKK, Kim}. The main result of \cite{BKK} is formulated for these curves.  As a final example, we consider {\it convenient} elliptic curves, a subclass of Tamagawa trivial curves with the guaranteed property
 (i.e. without computing $E(\Q)$)  that 
 $\widehat{h}(P)\geq \frac{1}{2}h_W(P)$ for all $P\in E(\Q),$
 where $\widehat{h}(P)$ (resp. $h_W(P)$) is the canonical (resp. Weil) height of $P$.
 
It is also natural  to consider curves with any fixed Tamagawa product $m$. As motivation, we recall that algorithms of Mazur, Stein, and Tate \cite{MST} for computing the global $p$-adic heights of rational points
on elliptic curves assume that the reductions of points at primes of bad reduction are non-singular. In follow-up work by Balakrishnan, \c{C}iperiani, and Stein \cite{BCS}, and Balakrishnan, \c{C}iperiani, Lang, Mirza, and Newton \cite{BCLMN}, where points can be defined over a more general number field, such assumptions can be computationally expensive. The knowledge of the proportion of
curves with arbitrary Tamagawa product $m$ gives an indication of the cost of such algorithms.

Motivated by these applications, we compute
the proportion of short Weierstrass elliptic curves 
\begin{equation}\label{Model}
E=E(a_4, a_6): \ \ y^2 = x^3 +a_4x +a_6,
\end{equation}
with $a_4, a_6\in \Z$ and $\Delta(a_4,a_6):=-16(4a_4^3+27a_6^2)\neq 0,$ with $\Tam(E)=m.$
To this end, we recall that $E$ has {\it height} 
\begin{equation}
\height(E):=\max\{4|a_4|^3, 27a_6^2\},
\end{equation}
and we employ the counting function
\begin{equation}\label{NN}
\NN(X):=\# \{E=E(a_4, a_6) \ : \ \ \height(E) \leq X\},
\end{equation}
which is the number of  $E(a_4,a_6)$ with height $\leq X.$
The number  with Tamagawa product $m$ is 
\begin{equation}\label{NNT}
\NN_{m}(X):=\# \{ E:=E(a_4,a_6) \ : \ \height(E)\leq X\ {\text {\rm with $\Tam(E)=m$}}\},
\end{equation}
where $\Tam(E)$ is the product for the global minimal model\footnote{By Lemma~\ref{Minimal}, the proportion of $E=E(a_4,a_6)$ with $\height(E)\leq X$ that are already minimal is {\color{black}$\rho=0.9960\dots$}.} of $E$. 
Our aim is to compute
\begin{equation}
P_{\Tam}(m):=\lim_{X\rightarrow +\infty}\frac{\NN_{m}(X)}{\NN(X)}.
\end{equation}
We compute the Dirichlet series generating function for these proportions.

 \begin{theorem}\label{Tam_m} The  $P_{\Tam}(m)$ are well-defined, and are the Dirichlet coefficients of
 $$
 L_{\Tam}(s):=\sum_{m=1}^{\infty}\frac{P_{\Tam}(m)}{m^s}=\prod_{p\ prime} \left(\frac{\delta_p(1)}{1^s}+\frac{\delta_p(2)}{2^s}+\frac{\delta_p(3)}{3^s}+\dots\right),
 $$
 where $\delta_p(\tp)$ are rational numbers defined in Lemma~\ref{Four}.
 \end{theorem}
 
 \begin{remark} The number $\delta_p(\tp)$ is  the proportion of  short Weierstrass curves whose minimal model
has $c_p=\tp.$
Cremona and Sadek \cite{CremonaSadek} compute such proportions for long Weierstrass models, which are different for $p\in \{2, 3\}.$
Our choice is motivated by an application to heights (see Corollary~\ref{MainCorollary}).
\end{remark}

\begin{corollary}\label{TamagawaTrivial} Assuming the notation above, the following are true.

\begin{enumerate}
\item We have that
$$P_{\Tam}(1)= \prod_{p\ prime} \delta_p(1) ={\color{black} 0.5053\dots},
$$
where $\delta_2(1)=241/396$, $\delta_3(1)=1924625/2125728$, and for primes $p\geq 5$ we have
$$
\delta_p(1)= \displaystyle{\  1-\frac{p(6p^7+9p^6+9p^5+7p^4+8p^3+7p^2+9p+6)}{6(p+1)^2(p^8+p^6+p^4+p^2+1)}}.
$$

 \item For primes $\ell$, we have
$P_{\Tam}(\ell)=\sum_{p} \delta_p(\ell)\prod_{\substack{q\neq p\\ prime}} \delta_q(1).$
\end{enumerate}
\end{corollary}

\begin{example}
These tables give $P_{\Tam}(1),\dots, P_{\Tam}(12),$ and show the convergence to $P_{\Tam}(1),$ $P_{\Tam}(2),$ and $P_{\Tam}(3).$
\begin{table}[h]
\begin{tabular}{|r|cc|cc|cc|cc|cc|cc|cc|cc|}
\hline
$m$    && $1$ && $2$ && $3$  && $4$ && $5$ && $6$ \\   \hline
$P_{\Tam}(m)$ && {\color{black}$0.5053\dots$} && {\color{black}$0.3391\dots$} && {\color{black}$0.0683\dots$} && {\color{black}$0.0622\dots$} && {\color{black}$7.98\dots\times 10^{-5}$} && {\color{black}$0.0158\dots$} \\
\hline
\hline 
$m$    && $7$ && $8$ && $9$ && $10$ && $11$ && $12$ \\   \hline
$P_{\Tam}(m)$ && {\color{black}$5.56\dots\times 10^{-6}$} && {\color{black}$0.0056\dots$} && {\color{black}$0.0011\dots$} && {\color{black}$4.56\dots \times 10^{-5}$} && {\color{black}$2.01\dots \times 10^{-7}$} && {\color{black}$0.0015\dots$} \\
\hline
\end{tabular} 
\smallskip
\caption{Proportions $P_{\Tam}(1),\dots, P_{\Tam}(12)$}
\end{table}

\begin{center}
\begin{table}[!ht]
\begin{tabular}{ | c | c |c|c| }
\hline
{$X$} & {$\NN_{1}(X)/\NN(X)$} & {$\NN_2(X)/\NN(X)$} &  {$\NN_3(X)/\NN(X)$} \\ \hline
$10^6$ & $0.5072\dots$  & $0.3384\dots$ & $0.0672\dots$ \\ \hline
$10^{8}$ & $0.5056\dots$  & $0.3389\dots$ & $0.0685\dots$ \\ \hline
$\vdots$ & $\vdots \ \ \ $ & $\vdots \ \ \ $& $\vdots \ \ \ $ \\ \hline
$\infty $ & ${\color{black}0.5053\dots}$  & ${\color{black}0.3391\dots}$ & ${\color{black}0.0683\dots}$ \\ \hline
\end{tabular}
\medskip
\caption{Convergence to $P_{\Tam}(1), P_{\Tam}(2)$ and $P_{\Tam}(3)$}
\end{table}
\end{center}
\end{example}

To place Theorem~\ref{Tam_m} in context, we recall that
work by Klagsbrun and Lemke-Oliver \cite{KLO}, and of Chan, Hanselman and Li \cite{CHL} shows that 
$\Tam(E)$ is unbounded in families of elliptic curves with prescribed $\Q$-rational 2-torsion. Moreover, Figure A.14 of \cite{DataBase}  suggests an  ``average Tamagawa product''
 of $\approx 1.82\dots,$ which we confirm using
 \begin{equation}
S_{\Tam(X)} :=\sum_{\height(E(a_4,a_6))\leq X} \Tam(E(a_4,a_6)),
\end{equation}
and the convergent special value of $L_{\Tam}(-1).$

\begin{theorem}\label{AverageTamagawa}
We have
$$
L_{\Tam}(-1)=\lim_{X\rightarrow +\infty} \frac{S_{\Tam}(X)}{\NN(X)}
= {\color{black}1.8193\dots}
$$
\end{theorem}

\begin{example}
Table 3 illustrates Theorem~\ref{AverageTamagawa}.

\begin{center}
\begin{table}[!ht]
\begin{tabular}{ | c || c | c | c | c | c |}
\hline
{$X$} & $10^4$ & $10^6$ & $10^8$  $\dots$ &$\dots$ & $\infty$ \\ \hline
{$S_{\Tam}(X)/\NN(X)$} & $1.8358\dots$  & $1.8291\dots$ & $1.8240\dots$ & $\dots$ & ${\color{black}1.8193\dots}$ \\ \hline
 \end{tabular}
\medskip
\caption{Convergence to average Tamagawa product}
\end{table}
\end{center}
\end{example}
\begin{remark} In \cite{DataBase} the authors only count a single curve in each isomorphism class of curves in short Weierstrass form. Counting in this way would would alter  $S_{\Tam}(X)$ and $\NN(X)$ by a factor approaching $\zeta(10)$ (as seen in the proof of Lemma \ref{Minimal}). However, the limiting ratio $L_{\Tam}(-1)$ would remain unchanged.
\end{remark}

Finally,
 we identify a subclass of elliptic curves that are convenient for height calculations.
 For $E=E(a_4,a_6),$ each 
 $P\in E(\Q)$ has the form $P=(\frac{A}{C^2},\frac{B}{C^3})$, with $A,B,C\in \Z,$ with $\gcd(A,C)=\gcd(B,C)=1$. The {\it naive height} of $P$ is
$H(P):=\max(|A|, |C^2|).$ 
The Weil height is 
$h_W(P):=\log H(P),$
and the {\it canonical height} is 
\begin{equation}
\widehat h(P):= \tfrac{1}{2}\lim_{n\to \infty}\frac{h_W(nP)}{n^2}.
\end{equation}

Logarithmic and canonical heights are generally close. 
Generalizing an observation of Buhler, Gross and Zagier \cite{BGZ}, we identify a natural subset of Tamagawa trivial curves that  automatically 
(i.e. without computing $E(\Q)$) have the property that 
$\widehat{h}(P)\geq \frac{1}{2}h_W(P)$ for every $P\in E(\Q)$ 
\begin{definition}
A Tamagawa trivial $E(a_4, a_6)/\Q$ is {\bf convenient} if it satisfies one of the following:
  \begin{enumerate}
 \item We have that $E(a_4, a_6)$ is a minimal model, and that $E(\R)$ has one connected component with
  $$ a_4\leq 0 \ \ \ {\text and}\ \ \ [\alpha,\infty)\subset\left\{x\in \R:2a_4x^2+8a_6x-a_4^2<0\right\},$$ where $\alpha$ is the real root of $x^3+a_4x+a_6$. 
  
 \item We have that $E(a_4, a_6)$ is a minimal model, and that
 $E(\R)$ has two connected components with
  $$a_4\leq 0 \ \ \ {\text and}\ \ \ [\gamma,\beta]\cup[\alpha,\infty)\subset\left\{x\in \R:2a_4x^2+8a_6x-a_4^2<0\right\},$$ where $\gamma<\beta<\alpha$ are the  real roots of $x^3+a_4x+a_6$.
 \end{enumerate}
\end{definition}

\begin{lemma}\label{Convenient}
 If $E(a_4,a_6)$ is convenient, then  the following are true.
 
 \smallskip 
 \noindent (1) If $E(\R)$ has one connected component, 
  then for all $P\in E(\Q)$ we have\[\widehat{h}(P)\geq \frac{1}{2}h_W(P).\]
 
 \smallskip  \noindent (2) If  $E(\R)$ has two connected components and  
 the smooth and continuous function $F_E(x)$ defined in \eqref{F_E} is positive for all $x\in (-1,1) \cap \big\{[\gamma,\beta]\cup [\alpha,\infty)\big\},$ then for all $P\in E(\Q)$ we have\[\widehat{h}(P)\geq \frac{1}{2}h_W(P).\]
\end{lemma}

As a corollary to Theorem~\ref{Tam_m}, we show that convenient curves  have a natural density using
\begin{equation}\label{NNc}
\NN_{c}(X):=\# \{ E(a_4,a_6) \ {\text {\rm convenient}} \ : \  \height(E(a_4,a_6))\leq X\}.
\end{equation}

\begin{corollary}\label{MainCorollary} 
 We have
$$
\lim_{X\rightarrow +\infty}\frac{\NN_c(X)}{\NN(X)}=\frac{12805748865(2+\sqrt{6})}{1830488\pi^{10}}\cdot P_{\Tam}(1)=0.1679\dots.
$$
\end{corollary}

\begin{example}
Table 4 illustrates Corollary~\ref{MainCorollary}. 
\begin{center}
\begin{table}[!ht]
\begin{tabular}{ | c || c | c | c | c | c |}
\hline
{$X$} & $10^5$ & $10^6$ & $10^7$  $\dots$ &$\dots$ & $\infty$ \\ \hline
{$\NN_c(X)/\NN(X)$} & $0.1741\dots$  & $0.1687\dots$ & $0.1678\dots$ & $\dots$ & ${\color{black}0.1679\dots}$ \\ \hline
 \end{tabular}
\medskip
\caption{Proportions of convenient curves}
\end{table}
\end{center}
\end{example}

The results obtained here are not difficult to derive. They follow from an analysis of Tate's algorithm. In Section~\ref{TateAlgorithm} we provide a
slight reformulation of the algorithm that is amenable for arithmetic statistics. In Section~\ref{AverageTamagawaProof} we prove  Theorems~\ref{Tam_m} and
\ref{AverageTamagawa} using this analysis. In Section~\ref{ConvenientLemma}
we prove Lemma~\ref{Convenient} by making use of standard facts about local height functions,
and by adapting a clever device of Buhler, Gross and Zagier \cite{BGZ}. Finally, in Section~\ref{CorollaryProof} we derive Corollary~\ref{MainCorollary} from Theorem~\ref{Tam_m} and Lemma~\ref{Convenient}.

\section*{Acknowledgements} \noindent  The second author thanks the NSF (DMS-2002265 and DMS-2055118) and
the  UVa Thomas Jefferson fund. The authors thank Jennifer Balakrishnan for useful discussions, and for pointing out the predicted average Tamagawa product.

\section{Tate's algorithm over $\Z$}\label{TateAlgorithm} Given a prime $p$,
Tate's algorithm \cite{SilvermanAdvanced, Tate} is a recursive procedure that determines the minimal model, conductor, Kodaira type, and the Tamagawa number $c_p$  of an elliptic curve. There are eleven steps, often involving changes of variable that produce simpler $p$-adic models. The algorithm can terminate at any of the first ten steps.  Curves that reach the eleventh step are not $p$-minimal. This eleventh step then applies the substitution $(x,y)\to (p^2x,p^3y),$
giving a $p$-integral model with discriminant that is reduced by a factor of $p^{12}.$ 
The resulting curve is inserted into the algorithm at step one, and the algorithm eventually  terminates due to the nonvanishing of discriminants.

We offer a reformulation of the algorithm  that is suited
for arithmetic statistics. Our goal is to compute $\delta'_p(K,\tp),$ the proportion of curves
$E(a_4,a_6)$ that are $p$-minimal, have Kodaira type $K,$ and Tamagawa number $c_p=\tp$.
We consider short Weierstrass curves $E=E(a_4,a_6)$, where $a_4, a_6\in \Z,$ with non-zero discriminant
$\Delta=\Delta(a_4,a_6):=-16(4a_4^3+27a_6^2).$ A model is {\it $p$-minimal} if the $p$-adic valuation of $\Delta$ is minimal among $p$-integral models. 
  The desired proportions, as well as some others, are summarized in tables in the Appendix.

\subsection{Classification for primes $p\geq 5$}\label{TateP5}

 For primes $p\geq 5,$
 the algorithm uses 
 five or seven numbers, which we refer to as the
 {\it Tate data for $E$ at $p$}
(i.e. $\alpha_4, \alpha_6, A_4, A_6, d\in \Z$ and $t, s\in \Z_p^{\times}$). The first five quantities are easily defined. We let $d:=v_p(\Delta),$ and we let  (note: $v_p(0):=+\infty)$
\begin{equation}\label{TwoOne}
\alpha_4:=v_p(a_4) \ \ \ {\text {\rm and}}\ \ \ \alpha_6:=v_p(a_6).
\end{equation}
Moreover, if $a_4\neq 0$ or $a_6\neq 0$, then $A_4$ and $A_6$ are defined by 
\begin{equation}\label{TwoTwo}
a_4=p^{\alpha_4}A_4 \ \ \ {\text {\rm and}}\ \ \ a_6=p^{\alpha_6}A_6.
\end{equation} 

We require the following $p$-minimality criterion, whose proof defines the invariants $s$ and $t$.

\begin{lemma}\label{Minimalp5} For primes $p\geq 5,$
$E(a_4, a_6)$
is non-minimal at $p$ if and only if $\alpha_4\geq 4$ and $\alpha_6\geq 6.$
\end{lemma}
\begin{proof}
Obviously, if $\alpha_6\geq 6$ and $\alpha_4\geq 4$, then the curve is non-minimal at $p$. 
To justify the converse, we note that
any counterexamples would satisfy $v_p(\Delta)\geq 12$ and
\begin{eqnarray}\label{BAD}\
v_p(4a_4^3)=v_p(27a_6^2)<v_p(\Delta),
\end{eqnarray}
due to the required cancellation of $p$-adic valuations.
Furthermore, (\ref{BAD})  implies that
$2\alpha_6=3\alpha_4$, which in turn implies that
$3\mid \alpha_6$ and $2\mid \alpha_4$.

For curves satisfying (\ref{BAD}), 
we define invariants $s$ and $t$ which are often required for Tate's algorithm.
As $v_p(\Delta)>0$, we have that
 $-3A_4$ is a quadratic residue modulo $p,$ and so there is
 a $p$-adic unit $t\in \Z_p^\times$ with $a_4=-3p^{\alpha_4}t^2.$
After a short calculation, which requires the correct choice of sign for $t$, we obtain a $p$-adic unit $s\in \Z_p^{\times}$ for which $a_6=2p^{\alpha_6}t^3+p^{d-\alpha_6} s.$
The utility of $s$ and $t$ arises from the singular point
 $(p^{\alpha_6/3}t,0)$ on 
$$
y^2\equiv x^3-3p^{\alpha_4}t^2 x +2p^{\alpha_6}t^3 \pmod{p^{d-\alpha_6}}.
$$
Under the substitution $x\to x+p^{\alpha_6/3}t$, it is mapped conveniently to $(0,0)$ on
\begin{equation}\label{subs_t}
y^2= x^3+3p^{\alpha_6/3}tx^2 +p^{d-\alpha_6}s.
\end{equation}
This substitution does not change the discriminant, and one can apply
 Tate's algorithm to this simpler model. For
 $\alpha_6\in \{0,3\},$  the algorithm terminates at one of the first ten steps, implying
 the  $p$-minimality of the original model.
The remaining cases
satisfying (\ref{BAD}) have $\alpha_6\geq 6$ and $\alpha_4\geq 4,$ thereby completing the proof.
\end{proof}

We now reformulate the algorithm (for $p\geq 5$) by identifying its steps with
one of eleven disjoint possibilities for $(\alpha_4, \alpha_6, d)$. 
The first ten cases, which correspond to $p$-minimal models, are in one-to-one correspondence with the possible Kodaira types. 
We compute  $\delta'_p(K,\tp)$ in each case. 
By Lemma~\ref{Minimalp5}, the eleventh case, where $\alpha_4\geq 4$ and $\alpha_6\geq 6,$ correspond to non-minimal short Weierstrass models.
We follow the steps as ordered in
Section IV.9 of  \cite{SilvermanAdvanced}, and so it would be convenient for the reader to have this reference readily available.

\smallskip
\noindent
{\bf Case 1} ($\boldsymbol{d=0}$). 
This case is Kodaira type $I_0,$ which has good reduction at $p.$ As $d=v_p(\Delta)=0,$ we have $(a_4, a_6)\not \equiv (-3w^2, 2w^3)\pmod{p},$ for any $w\in \F_p.$ Therefore, there are $p^2-p$ choices of $(a_4,a_6)$ out of $p^2$ total possible pairs modulo $p$, and so Tate's algorithm gives $\delta_p'(I_0,1)=(p-1)/p.$

\smallskip
\noindent
{\bf Case 2} ($\boldsymbol{\alpha_4=\alpha_6=0<d}$). This case is Kodaira type $I_{n\geq 1},$ where $0<n=d=v_p(\Delta).$  
We require (\ref{BAD}) (i.e. cancellation of $p$-adic valuations in the discriminant), and we employ the discussion in the proof of Lemma~\ref{Minimalp5}, where $(a_4,a_6)=(-3t^2,2t^3+p^ds).$ If we let
$\varepsilon(n):=((-1)^n+3)/2,$ then the algorithm gives
$$c_p:=\frac{n\left(1+\leg{3t}{p}\right)}{2}+\frac{\varepsilon(n)\left(1-\leg{3t}{p}\right)}{2}.$$

For each $d\geq 1$, we consider $(a_4,a_6)\pmod{p^{d+1}}$. Since $p\nmid t,$ there are $p^{d+1}-p^d$ choices of $t.$ We also have $p-1$ choices of $s\not \equiv 0\pmod{p}$. This gives us $p^d(p-1)^2$ choices of $(a_4,a_6)$ out of $p^{2d+2}$ total possible pairs modulo ${p^{d+1}}$. If $n>2,$ then half of these choices will have $c_p=n$, and the other half will have $c_p=1$ or $2$.  
Therefore, we have that
$\delta_p'(I_1,1)=(p-1)^2/p^3,$ $\delta_p'(I_2,2)=(p-1)^2/p^4,$ and $\delta_p'(I_{n\geq 3},n)=
 \delta_p'(I_{n\geq 3},\varepsilon(n))=(p-1)^2/2p^{n+2}.$

\smallskip
\noindent
{\bf Case 3} ($\boldsymbol{\alpha_4\geq 1 \textbf{ and } \alpha_6=1}$).   This is Kodaira type $II,$ with $c_p=1.$  There are $p$ choices of $a_4\pmod{p^2}$ and $p-1$ choices of $a_6 \pmod{p^2},$ giving $p(p-1)$ many options from $p^4$ possibilities. Therefore, we have $\delta_p'(II,1)=(p-1)/p^3.$

\smallskip
\noindent
{\bf Case 4} ($\boldsymbol{\alpha_4= 1 \textbf{ and } \alpha_6\geq 2}$).  This case is Kodaira type $III,$ where $c_p=2.$ There are $p-1$ choices of $a_4\pmod{p^2},$ and $1$ choice of $a_6 \pmod{p^2}.$
As there are $p^4$ many possible pairs, we have $\delta_p'(III,2)=(p-1)/p^4.$

\smallskip
\noindent
{\bf Case 5} ($\boldsymbol{\alpha_4\geq 2 \textbf{ and } \alpha_6= 2}$).  This case is Kodaira type $IV,$ where $c_p\in \{1, 3\}.$  As $a_6=p^2A_6,$ Tate's algorithm gives $c_p:= 2+\left(\frac{A_6}{p}\right).$ 
There are $p$ choices of $a_4\pmod{p^3}$ and $p-1$ choices of $a_6 \pmod{p^3},$ and $p^6$ possible pairs modulo  ${p^3}$. Half of these pairs have $c_p=1$ (resp. $c_p=3$). Therefore, we have $\delta_p'(IV,1)=\delta_p'(IV,3)=(p-1)/2p^5.$

\smallskip
\noindent
{\bf Case 6} ($\boldsymbol{\alpha_4\geq 2, \alpha_6\geq 3 \textbf{ and }   d=6}$). This case is Kodaira type $I_0^*,$ where $c_p\in\{1,2,4\}$.  
This step of Tate's algorithm requires an auxiliary polynomial $P(T),$ which we now define. Given a long Weierstrass model 
 \begin{equation}\label{LongModel}
 y^2+\underbar{a}_1 xy +\underbar{a}_3y = x^3+\underbar{a}_2x^2+\underbar{a}_4x+\underbar{a}_6,
 \end{equation}
 (possibly after a change of variable) so that $p\mid \underbar a_2,$ $p^2\mid \underbar a_4,$ and $p^3\mid \underbar a_6,$ we define
 \begin{equation} \label{Def_P}
P(T):=T^3+p^{-1}a_2T^2+ p^{-2}a_4T +p^{-3}a_6.
 \end{equation}
As we are working with short Weierstrass models, we have $P(T)=T^3+a_4p^{-2}T +a_6p^{-3},$ and its discriminant is $ 2^{-4}p^{-6}\Delta$. Since $d=6$, we see that $P(T)$ has distinct roots modulo $p$.  In this case, the algorithm states that the Tamagawa number is $1$ more than the number of roots of $P(T) $ in $\F_p$ (the finite field with $p$ elements). 

To calculate $\delta_p'(I_0^*,1),$ we count the number of trace $0$ separable cubics over $\F_p.$ 
We first consider the number of choices of $P(T)$ which are irreducible, corresponding to $c_p=1$. There are $p^3-p$ elements of $\F_{p^3}$ not in $\F_p$, and $1/p$ of those have trace $0$. Thus there are $(p^2-1)/3$ possible choices for $P(T)$ to be irreducible. We next consider the possibility that $P(T)$ factors as $Q(T)(T-\alpha)$, with $Q(T)$ an irreducible quadratic, corresponding to $c_p=2$. In this case $\alpha$ is uniquely determined by $Q(T)$ so that the trace of $P(T)$ is $0$. There are $(p^2-p)/2$ irreducible quadratics modulo $p$, and therefore $(p^2-p)/2$ possible choices of $P(T)$ with a single root in $\F_p$. Finally we consider the case that $P(T)$ factors completely in $\F_p$ with distinct roots, corresponding to the case $c_p=4$. There are ${p}\choose {3}$ ways of choosing $3$ distinct roots in $\F_p$, and $1/p$ of those have trace $0$. Thus, there are $(p-1)(p-2)/6$ choices for $P(T)$ in this case.  There are $p^3$ total possible choices of $a_4 \pmod{p^3}$ and $p^4$ total possible choices of $a_6 \pmod{p^4}.$ 
Therefore, we 
find that $\delta_p'(I_0^*,1)=(p^2-1)/3p^7,$ $\delta_p'(I_0^*,2)=(p-1)/2p^6,$ and
$\delta_p'(I_0^*,4)=(p-1)(p-2)/6p^7.$

\smallskip
\noindent
{\bf Case 7} ($\boldsymbol{\alpha_4=2, \ \alpha_6=3 \textbf{ and } d>6}$).
This case is Kodaira type $I^*_{n\geq 1 }$ with $c_p\in \{2,4\},$ where $0<n:=d-6.$  We employ model (\ref{subs_t}), where $a_4=-3 p^2\,t^2$ and $a_6=2p^3t^3+p^{d-3}s.$ Tate's algorithm requires the polynomial $P(T)$ defined in (\ref{Def_P}), as well as an additional  auxiliary polynomial which we denote by $R(Y).$ Given a long Weierstrass model (\ref{LongModel}) (possibly after a change of variable) so that $p^2\mid \underbar a_3,$ and $p^4\mid \underbar a_6,$ we define
 \begin{equation} \label{Def_R}
R(Y):=Y^2+p^{-2}a_3Y- p^{-4}a_6.
 \end{equation}
As our curves are in short form, we have $P(T)=T^3+3tT^2 +p^{n}s,$ and $R(Y)=Y^2-p^{n-1}s.$  The Tamagawa number $c_p$ depends on $P(T)$ and $R(Y)$. 
  If $n$ is odd, then we consider the roots of $R(Y)$. We make the substitution $p^{-n+1}R(p^{\frac{n-1}{2}}Y)=Y^2-s$ which corresponds to the application of the sub-procedure of Step $7$ of Tate's algorithm $n$ times. The number of roots depends on whether or not $s$ is a quadratic residue. The  algorithm gives $c_p:= 3+\left(\frac{s}{p}\right).$
 If $n$ is even, then we consider the number of roots of $P(T)$. We make the substitution $p^{-n}P(p^{\frac{n}{2}}T)\equiv 3tT^2+s\pmod{p}$ which corresponds to $n$ steps through the sub-procedure of Step $7$ of Tate's algorithm. The number of roots depends on whether or not $-s/3t$ is a quadratic residue. The algorithm gives that $c_p:=3+\left(\frac{-s/3t}{p}\right).$ 
 
To calculate the proportion of curves satisfying this condition for a given $n\geq 1,$ we note that $(a_4,a_6) \pmod{p^{n+4}}$ is determined by any choice of $t \pmod{p^{n+2}}$ and $s\pmod{p},$ where  $p\nmid st.$ Half of the possible choices of $s$ correspond to each $c_p$. Therefore, we obtain  $\delta_p'(I^*_{n},2)=\delta_p'(I^*_{n },4)=(p-1)^2/2p^{7+n}.$

\smallskip
\noindent
{\bf Case 8} ($\boldsymbol{\alpha_4\geq 3 \textbf{ and } \alpha_6= 4}$).  
This is Kodaira type $IV^*$, with $c_p\in \{1, 3\}$.  As $a_6=p^4A_6,$ the algorithm gives $c_p:=2+\left(\frac{A_6}{p}\right).$ 
There are $p^2$ choices of $a_4\pmod{p^5}$ and $p-1$ choices of $a_6 \pmod{p^5}$, and there are $p^{10}$ possible pairs. Half of these pairs have $c_p=1$ (resp. $c_p=3$), and so
we obtain $\delta_p'(IV^*,1)=\delta_p'(IV^*,3)=(p-1)/2p^8.$

\smallskip
\noindent
{\bf Case 9} ($\boldsymbol{\alpha_4= 3 \textbf{ and } \alpha_6\geq 5}$).  This case is Kodaira type $III^*,$ where $c_p=2$. There are $p(p-1)$ many choices of $a_4\pmod{p^5},$ and $1$ choice of $a_6 \pmod{p^5}.$ As there are $p^{10}$ possible pairs modulo ${p^5},$ we obtain
$\delta_p'(III^*,2)=(p-1)/p^9.$

\smallskip
\noindent
{\bf Case 10} ($\boldsymbol{\alpha_4\geq 4 \textbf{ and } \alpha_6=5}$).  This case is Kodaira type $II^*,$ where $c_p=1$.  This case depends on $(a_4,a_6)\pmod{p^{6}}$. There are $p^2$ choices of $a_4\pmod{p^6}$ and $p-1$ choices of $a_6 \pmod{p^6}.$  As there are $p^{12}$ possible pairs modulo ${p^6},$ we obtain
$\delta_p'(II^*,1)=(p-1)/p^{10}.$

\smallskip
\noindent
{\bf Case 11} ($\boldsymbol{\alpha_4\geq 4 \textbf{ and } \alpha_6\geq 6}$). As the model is not minimal,  the algorithm replaces $a_4$ and $a_6$ with $a_4/p^{4}$ and $a_6/p^{6}$ respectively. One repeats these substitutions until one obtains a model which is one of the ten cases above.

\subsection{Classification for $p=3$}
The Tate data for $p=3$ also consists of five or seven numbers. The first five (i.e. $\alpha_4, \alpha_6, A_4, A_6,$ and $d$) are defined by   (\ref{TwoOne}) and (\ref{TwoTwo}). 
The next lemma classifies those $E(a_4,a_6)$ that are not 3-minimal, and its proof 
defines invariants $s$ and $t$ in some of the cases where further invariants are required.

\begin{lemma}\label{Minimalp3}  The curve
$E(a_4, a_6)$
 is not 3-minimal if and only if  $(\alpha_4, \alpha_6, d)$ satisfies one of the following conditions:
\begin{enumerate}
\item We have that $\alpha_4\geq4$ and $\alpha_6\geq 6.$
\item We have that $\alpha_4=\alpha_6=3$ and $d\geq 12$. 
\end{enumerate}
\end{lemma}
\begin{proof}
Obviously, $E(a_4, a_6)$ is not minimal at $3$ when (1) holds. 
Therefore, to prove the lemma, suppose that $E(a_4,a_6)$ does not satisfy (1) and is not 3-minimal.
Then, we have $v_3(\Delta)\geq 12$ and
\begin{equation}\label{BAD3}
v_3(4a_4^3)=v_3(27a_6^2) < v_3(\Delta),
\end{equation}
due to the necessary cancellation of $3$-adic valuations.  Furthermore, (\ref{BAD3})  implies that
$2\alpha_6+3=3\alpha_4$, which in turn implies that
$3\mid \alpha_6$ and $\alpha_4$ is odd.
Arguing precisely as in the previous subsection, we find that for any curve satisfying (\ref{BAD3}) there are $3$-adic units $s,t\in \Z_3^{\times}$ for which
$$ a_4 = -3^{\alpha_4}t^2  \ \ \ \text{and} \ \ \ a_6=2\cdot 3^{\alpha_6}t^3+3^{d-\alpha_6-3} s.$$
The substitution $x\to x+3^{\alpha_6/3} t$ returns the model
\begin{equation}\label{model1}
y^2=x^3+3^{1+\alpha_6/3} t\cdot x^2+3^{d-\alpha_6-3}s,
\end{equation}
which has the same discriminant. The assumption that $E(a_4,a_6)$ does not satisfy (1) implies that $\alpha_6\in \{0,3\}$. 
For $\alpha_6=0$, the algorithm applied to this model terminates at one of the first $7$ steps, and so the original model is $3$-minimal. However, if $\alpha_4=\alpha_6=3,$ and $d\geq 12,$ we see that (\ref{model1}) is not minimal. The additional substitution $(x,y)\to (9x,27y)$ returns the reduced model 
\begin{equation}\label{Reduced3}
y^2=x^3+tx^2+3^{d-12}s,
\end{equation}
with smaller discriminant $3^{-12}\Delta.$
 This completes the proof of the lemma.
\end{proof}

We also require two further invariants $s$ and $t$ when
 $\alpha_6\in \{0,3\},$ and
$v_3(4a_4^3)\geq v_3(27a_6^2)=v_3(\Delta).$
We note that if $v_3(4a_4^3)= v_3(27a_6^2)=d,$ then  $A_4\equiv 1 \pmod{3}.$ 
A straightforward calculation with Hensel's lemma shows that there is a $3$-adic unit $t\in \Z_3^\times,$ and $s\in \{0, \pm 1\}$ for which $a_6=  3^{\alpha_6}t^3+3^{\alpha_6/3}a_4 t+3^{\alpha_6+1}s.$ 
The substitution $x\to x-3^{\alpha_6/3} t$ gives the new model
\begin{equation}\label{model2}
y^2=x^3-3^{\alpha_6/3+1} t\cdot x^2+(a_4+3^{\frac{2}{3}\alpha_6+1}t^2)\cdot x+3^{\alpha_6+1}s.
\end{equation}
In this situation we shall apply the algorithm to this model.

As in the previous subsection, we reformulate the algorithm for $p=3$ by identifying its steps with
one of a number of disjoint possibilities for $(\alpha_4, \alpha_6, d)$.  Unlike the situation for primes $p\geq 5$, a few further cases  arise due to the fact that $3-$minimal models in short Weierstrass form do not always exist for $E(a_4,a_6)$. {\color{black} If $E(a_4,a_6)$ is minimal it will fall into one of cases (1)-(10). Non-minimal curves satisfying condition (1) of Lemma \ref{Minimalp3} will fall into case (11) and iterate through the algorithm again as in the case where $p\geq 5$, while non-minimal curves satisfying condition (2) of Lemma \ref{Minimalp3} will fall into extra cases designated (1*) or (2*).
} 
We let $\widehat{\delta}_3(K,\tp)$  be the proportion of curves which fall into these cases.

\smallskip
\noindent
{\bf Case 1} ($\boldsymbol{\alpha_4=0}$). 
This case is Kodaira type $I_0$, when we have good reduction at $p=3.$ We have $\alpha_4=0$ and $v_3(\Delta)=0.$ There are $2$ choices for $a_4$ modulo $3$, and so $\delta_3'(I_0,1)=2/3.$

\smallskip
\noindent
{\bf Case 1*} ($\boldsymbol{\alpha_4=\alpha_6=3, \textbf{ and }  d =12}$). 
This case is also for Kodaira type $I_0$, but  only arises in situations where the original $E$ is not $3$-minimal. Therefore, we employ model (\ref{Reduced3}) per the discussion above.
Namely, we have $a_4=-3^3t^2$, and $a_6=2\cdot 3^3t^3+3^6s$, with $3\nmid st,$ and so the  change of variable reduces $E$ to the $3$-minimal model  $y^2=x^3+tx^2+s,$ which has discriminant $-16(4t^3s+27s^2)\not \equiv 0 \pmod{3}.$ There are $18$  choices for $t$ modulo $27$ with $t\not \equiv 0 \pmod{3}$, and for each choice of $t$, there are $2$ choices of $s$ modulo $3$. Together, these determine $a_4$ modulo $3^6$ and $a_6$ modulo $3^7$
Combining these observations, we obtain $\widehat \delta_3(I_0,1)=4/3^{11}.$

\smallskip
\noindent
{\bf Case 2} ({\bf None}).
 This case is for Kodaira types $I_{n\geq 1}.$   Since $(x+c)^3=x^3+3cx^2+\dots,$ the substitutions $x\to x+c,$ with $c\in \Z$, always produces models with $3\mid b_2,$ where $b_2:=\underbar{a}_1^2+4\underbar{a}_2$
 for long models  as in (\ref{LongModel}).
 Tate's algorithm, which employs $b_2$, bypasses these cases for short models when $p=3$, and so
 we have $\delta_3'(I_{n\geq 1},c_3)=0.$

\smallskip
\noindent
{\bf Case 2*} ($\boldsymbol{\alpha_4=\alpha_6=3 \textbf{ and } d >12}$). This case concerns Kodaira types $I_{n\geq 1}$, when $E$ is not $3-$minimal. Following the discussion above, we employ (\ref{Reduced3}), where  $a_4=-3^3t^2$ and $a_6=2\cdot 3^3t^3+3^{d-6}s,$ with $3\nmid st$. The new model $y^2=x^3+tx^2+3^{d-12}s$ is $3$-minimal, and has discriminant $-16(4t^33^{d-12}s+3^{2d-24}s^2) \equiv 0 \pmod{3^{d-12}}.$ 
Tate's algorithm gives type $I_n$, with $n=d-12,$ the $3$-adic valuation of this discriminant. 
Moreover, the Tamagawa number $c_3$ depends on $t$ modulo $3,$ and we find that
$$c_3:=\frac{n\left(1+\leg{t}{3}\right)}{2}+\frac{\varepsilon(n)\left(1-\leg{t}{3}\right)}{2}.$$
Here $\varepsilon(n)$ is as in Case 2 of the previous subsection.
If we fix $n\geq 1,$ and $t\pmod{3}$, then $a_4\pmod{3^4}$ is uniquely determined. Furthermore, for each choice of $a_4$, there are $2$ choices of $a_6\pmod{3^{n+7}}$. 
Combining these facts, we obtain $\widehat \delta_3(I_1,1)=4/3^{12},$ $\widehat \delta_3(I_2,2)=4/3^{13},$ and $\widehat \delta_3(I_{n\geq 3},c_3)=2/3^{n+11}.$

\smallskip
\noindent
{\bf Case 3}. This case is Kodaira type $II,$ where $c_3=1.$   The following three possibilities for this case are:
\begin{enumerate}[label=(\roman*)]
\item We have $\alpha_4\geq 1$ and $\alpha_6=1.$
\item We have $\alpha_4\geq 1, \alpha_6=0$, $d=3,$ and $s\neq 0.$ Note that $s$ is from model (\ref{model2}).
\item We have $\alpha_4=1, \alpha_6=0$, and $d=4.$
\end{enumerate}
There are $6$ pairs $(a_4,a_6)$ modulo $9$ satisfying (i), and so their proportion is $6/81.$
For curves satisfying (ii), we use model (\ref{model2}). In this situation we have $a_6=3^{\alpha_6}t^3+3^{\alpha_6/3}a_4 t +3^{\alpha_6 +1}s.$ We have 
$2$ choices for $a_4\equiv 0,3\pmod{9}$, $2$ choices for $t\equiv \pm1  \pmod{3}$, and $2$ choices for $s=\pm 1,$ which together determine $a_6$ modulo $9$. Therefore their proportion is $8/81.$ 
For curves satisfying (iii), we use model (\ref{model1}).  We have $a_4=-3t^2$ and $a_6=2t^3+3s.$ There are $2$ choices each for $t, s\equiv \pm1 \pmod{3},$ which together determine $(a_4,a_6)$ modulo $9$. Therefore, their proportion is $4/81.$
Combining these observations, we obtain $\delta_3'(II,1)=\frac{2}{9}.$

\smallskip
\noindent
{\bf Case 4}.  This case is Kodaira type $III,$ where $c_3=2$. The following two possibilities for this case are:
\begin{enumerate}[label=(\roman*)]
\item We have $\alpha_4=1$ and $\alpha_6\geq 2.$
\item We have $\alpha_4\geq 1, \alpha_6=0$, $d=3,$ and $s= 0.$ Note that $s$ is from model (\ref{model2}).
\end{enumerate}
There are $2$ pairs $(a_4,a_6)$ modulo ${9}$ satisfying (i).
  For curves satisfying (ii), we use model (\ref{model2}). In this situation, we have $2$ choices for $a_4\equiv 0$ or $3\pmod{9}$, and $2$ choices for $t\equiv \pm 1\pmod{3},$ which together determine $a_6$ modulo $9$.  Hence, there are $4$ pairs $(a_4,a_6)$ modulo ${9}$ satisfying (ii).
Overall, the 6 pairs $(a_4,a_6)\in \left\{(3,0),(6,0),(0,\pm 1),(3,\pm 4) \pmod{9}\right\}$, chosen from $81$ possibilities, gives
 $\delta_3'(III,2)=2/27.$

\smallskip
\noindent
{\bf Case 5}.  This case is Kodaira type $IV,$ where $c_3\in \{1, 3\}.$ The following two possibilities for this case are:
\begin{enumerate}[label=(\roman*)]
\item We have $\alpha_4\geq2$ and $\alpha_6= 2.$
\item We have $\alpha_4=1,$ $\alpha_6=0,$ and $d=5.$
\end{enumerate}
For condition (i), Tate's algorithm gives $c_3:=2+\left(\frac{A_6}{3}\right).$ For condition (ii), the algorithm gives $c_3:=2+\left(\frac{s}{3}\right),$ where $s$ corresponds to model (\ref{model1}).

It is simple to determine the proportions of curves in these cases.
In case of (i),
we have $9\mid a_4,$ and 1 choice of $a_6$ modulo $27$ for each $c_3$, giving a proportion of $3^{-5}$ each of the two possible Tamagawa numbers. Curves satisfying (ii) are cases of (\ref{model1}), and so we have $a_4=-3t^2$ and $a_6=2t^3+9s.$ There are $6$ choices $t$ modulo $9,$ and one choice of $s$ modulo $3$ for each $c_3$. These together determine the $6$ choices of $a_4$ and $a_6$ modulo $27$ which fall under this set of conditions for each $c_3$. 
Combining these observations, we obtain $\delta_3'(IV,1)=\delta_3'(IV,3)=1/81.$

\smallskip
\noindent
{\bf Case 6}.
This case is  Kodaira type $I_0^*,$ where $c_3\in\{1,2,4\}.$ The following two possibilities for this case are:
\begin{enumerate}[label=(\roman*)]
\item We have $\alpha_4=2$ and $\alpha_6\geq 3.$
\item We have $\alpha_4=1,$ $\alpha_6=0,$ and $d=6.$
\end{enumerate}
In each case, $c_3$ is $1$ more than the number of roots of the polynomial $P(T)$ defined in (\ref{Def_P}), possibly after a change of variable. 
Assuming (i), we have 
$P(T)= T^3 +A_4T +p^{\alpha_6-3} A_6, $ This polynomial has $1$ root modulo $3$ if $A_4\equiv 1\pmod{3}$, $3$ distinct roots if $A_4\equiv -1\pmod{3}$ and $\alpha_6>3$, and $0$ roots modulo $3$ if $A_4\equiv -1\pmod{3}$ and $\alpha_6=3.$  
For curves satisfying (ii), we  define $P(T)$ using model (\ref{model1}). We have $a_4=-3t^2$ and $a_6=2t^3+27 s,$ where $s$ and $t$ are $3$-adic units, and 
$P(T)= T^3 +tT^2 + s. $ This polynomial has $1$ root modulo $3$ if $t\equiv s \pmod{3},$ and $0$ roots modulo $3$ if $t\not \equiv s \pmod{3}.$ 
Each $(a_4,a_6) \pmod{81}$ corresponds to a pair $t$ modulo ${27}$ and $s$ modulo $3$. 
Therefore, we have that $c_p=1$ for $24$ pairs of $(a_4,a_6)$ modulo ${81}$ ($6$ pairs satisfying (i), and $18$ satisfying (ii))  $c_p=2$ for $27$ pairs of $(a_4,a_6)$ modulo $81$ ($9$ pairs satisfying (i), and $18$ satisfying (ii)), and $c_p=4$ for $3$ pairs of $(a_4,a_6)$ modulo $81$ (all satisfying condition (i)). 
Hence, we have $\delta_3'(I_0^*,1)=8/3^7,$ $\delta_3'(I_0^*,2)=1/3^5,$ and $\delta_3'(I_0^*,4)=1/3^7.$

\smallskip
\noindent
{\bf Case 7} ({$\boldsymbol{\alpha_4\geq 1,\alpha_6=0 \textbf{ and } {\color{black}d\geq7}}$).}
This case is Kodaira type $I^*_{n\geq 1 },$ with $c_3\in\{2,4\},$ where $0<n:=d-6.$ We use model (\ref{model1}), and define the auxiliary polynomials $P(T)$ and $R(Y)$ as in (\ref{Def_P}) and (\ref{Def_R}).  We have that $P(T)= T^3 +tT^2 + 3^{n}s $ and  $R(Y)=Y^2-3^{n-1}s$. 

If $d$ is odd, then we have 
\[3^{-n+1}R(3^{\frac{n-1}{2}}Y)\equiv Y^2-s \pmod{3},\]
which follows from the application of the sub-procedure in Step 7 of the  algorithm $n$ times. We have $c_3:=4$ when $s\equiv 1\pmod{3}$ (so that $R(Y)$ factors over $\Z_3$), and
$c_3:=2$ when $s=-1 \pmod{3}$ (so that $R(Y)$ does not factor).

If $d$ is even, then we have that 
$3^{-n}P(3^{\frac{n}{2}}T)\equiv tT^2+s\pmod{3}$, which  
corresponds to $n$ steps through the sub-procedure of Step 7 of Tate's Algorithm. We have $c_3:=4$ if $s\equiv -t\pmod{3}$ (so that $P(T)$ factors completely), and $c_3:=2$ if $s\equiv t\pmod{3}$ (so that $P(T)$ does not factor completely).

For each $n\geq 1$ and each $c_3\in \{2,4\}$, we have $2\cdot 3^{n+2}$ choices for $t$ modulo $3^{n+3}$ and $1$ choice of $s$ modulo $3$ (depending on $c_3$ and $t$). This determines 
$2\cdot 3^{n+2}$ pairs $(a_4,a_6)$ modulo $3^{n+4}$ corresponding to Kodaira type $I^*_{n}$ with the chosen Tamagawa number $c_3,$ out of a total of $3^{2n+8}$ pairs modulo $3^{n+4}.$ Hence, Tate's algorithm gives $\delta_3'(I_n^*,2)= 2/3^{n+6},$ and $\delta_3'(I_n^*,4)=2/3^{n+6}.$

\smallskip
\noindent
{\bf Case 8}. This case is Kodaira type $IV^*$, where $c_3\in \{1, 3\}.$ The following two possibilities for this case are:
\begin{enumerate}[label=(\roman*)]
\item We have $\alpha_4\geq3$ and $\alpha_6=4.$
\item We have $\alpha_4\geq 3, \alpha_6=3$, $d=9,$ and $s\neq 0.$ Note that $s$ is from model (\ref{model2}).
\end{enumerate}
Tate's algorithm gives $c_3:=2+\left(\frac{A_6}{3}\right)$ under condition (i), and $c_3:=2+\left(\frac{s}{3}\right)$ under condition (ii). For condition (i), we have $1$ choice for $a_4$ modulo $27,$ and one choice for $a_6$ modulo $3^{5}$ for each of the two possible values of $c_3$ (i.e. a proportion of $1/3^8$). For condition (ii), there are $2$ choices for  $a_4\equiv 0 $ or $27 \pmod{3^4},$ $2$ choices for $t\not \equiv 0 \pmod 3,$ and $1$ choice of $s\neq0$ for each $c_3.$ Together, these determines $a_6$ modulo $3^5.$ These pairs $(a_4,a_6)$ occur with proportion $4/3^{9}$ for each of the two possible values for $c_3.$ Therefore, we obtain $\delta_3'(IV^*,1)=\delta_3'(IV^*,3)=7/3^9.$

\smallskip
\noindent
{\bf Case 9}.  This case is Kodaira type $III^*,$ where $c_3=2.$ 
The following three possibilities for this case are:
\begin{enumerate}[label=(\roman*)]
\item We have $\alpha_4=3$ and $\alpha_6\geq 5.$
\item We have $\alpha_4\geq 3, \alpha_6=3$, $d=9,$ and $s= 0.$ Note that $s$ is from model (\ref{model2}).
\item We have $\alpha_4=\alpha_6=3$ and $d=10.$
\end{enumerate}
The proportion of curves satisfying (i) is $2/3^{9}$.
 For curves satisfying (ii), we use (\ref{model2}). In this situation, we find that  $a_4\equiv 0, 27 \pmod{3^4}$. Thus, there are 2 choices of $a_4$ modulo $3^4,$ and 2 choices for $t \equiv \pm1 \pmod 3.$ These choices determine $a_6$ modulo $3^5.$ Therefore, the proportion of curves satisfying this set of conditions is $4/3^{9}.$
For curves satisfying (iii), we use (\ref{model1}). In this situation, we have $a_4=-27t^2$ and $a_6=54t^3+81s.$ There are $2$ choices each for $t, s\equiv \pm 1 \pmod{3}.$ These choices together determine $a_4$ modulo $3^4$ and $a_6$ modulo $3^5$. Therefore, the proportion of these curves is $4/3^9.$
Combining these observations, we obtain $\delta_3'(III^*,2)=10/3^9.$

\smallskip
\noindent
{\bf Case 10}.
 This case is Kodaira type $II^*,$ where $c_3=1.$ 
 The following two possibilities for this case are:
\begin{enumerate}[label=(\roman*)]
\item We have $\alpha_4\geq 4$ and $\alpha_6= 5.$
\item We have $\alpha_4=\alpha_6=3$, and $d=11.$
\end{enumerate}
 The proportion of curves satisfying (i) is $2/3^{10}.$
 For curves satisfying (ii), we use model (\ref{model1}). In this situation, there are $6$ choices for $t$ modulo $9$ so that $t\neq 0 \pmod{3},$ and $2$ choices for $s \equiv \pm 1 \pmod{3}.$ These choices together determine $a_4$ modulo $3^5$ and $a_6$ modulo $3^6.$ Therefore, the proportion of pairs $(a_4,a_6)$ satisfying these conditions is $4/3^{10}$
Combining these observations, we obtain  $\delta_3'(II^*,1)=2/3^{9}.$

\smallskip
\noindent
{\bf Case 11} ($\boldsymbol{\alpha_4\geq 4 \textbf{ and } \alpha_6\geq 6}$). As the model is not minimal,  the algorithm replaces $a_4$ and $a_6$ with $a_4/3^{4}$ and $a_6/3^{6}$ respectively. One repeats these substitutions until one obtains a model which is one of the ten cases above.

\subsection{Classification for $p=2$}\label{TwoClassification}
The Tate data for $p=2$ also consists of five or seven numbers. The first five (i.e. $\alpha_4, \alpha_6, A_4, A_6,$ and $d$) are defined by   (\ref{TwoOne}) and (\ref{TwoTwo}). 
The next lemma classifies those $E(a_4,a_6)$ that are not 2-minimal, and its proof 
defines invariants $s$ and $t$ in some of the cases where these invariants are required. 

\begin{lemma}\label{Minimalp2}  A curve
$E(a_4, a_6)$
 is not 2-minimal if and only if  $(a_4, a_6, d)$ satisfies one of the following conditions:
\begin{enumerate}
\item We have that $\alpha_4\geq4$ and $\alpha_6\geq 6.$
\item We have that $\alpha_4\geq 4$, $a_6\equiv 16 \pmod{64}$.
\item We have $(a_4,a_6)\equiv (5,6)\pmod{8}$, where $d\geq12$
\end{enumerate}
\end{lemma}
\begin{proof}
Obviously, $E(a_4, a_6)$ is not minimal at $2$ when (1) holds. 
Now, suppose that $E(a_4,a_6)$ does not satisfy (1) and is not 2-minimal.
Then, we have $v_2(\Delta)\geq 12$. This implies that either $a_4\geq 2$ and $a_6\geq4,$ or 
\begin{equation}\label{BAD2}
v_2(4a_4^3)=v_2(27a_6^2) < v_2(\Delta),
\end{equation}
due to the necessary cancellation of $2$-adic valuations. If $\alpha_4 =2$ (resp. $3$) and $\alpha_6\geq 4,$ then we find that Tate's algorithm terminates in Step 7 (resp. Step 8). If $\alpha_4\geq 4$ and $a_6 \equiv 48$ or $32\pmod{64},$ then Tate's algorithm terminates at Step 10. However, if $a_6\equiv 16\pmod{64},$  then we have $a_4=4A_4$ $a_6=16+2^ks,$ where either $s=0$ or $s$ odd, and $k\geq6.$ The substitution $(x,y)\to (4x, 8y+4),$ reduces the equation of the curve to 
\begin{equation}\label{Reduced2.1}
 y^2+y=x^3 +A_4x+2^{k-6}s,
 \end{equation}
which has discriminant $2^{-12}\Delta$.  Since $\alpha_6=4$, we have $v_2(\Delta)=12$, and so this model is $2$-minimal, giving
(2).

 If $E(a_4,a_6)$ satisfies (\ref{BAD2}), then $2\alpha_6=3\alpha_4+2$, which in turn implies that
$\alpha_4$ is even. We find that $\alpha_4\in \{0,2\},$ however the possible case that $\alpha_4 =2$ and $\alpha_6=4$ was already considered above. If $\alpha_4=0,$ then $\alpha_6=1.$ For any curve of this type satisfying (\ref{BAD2}) with $v_2(\Delta)\geq 12,$ we find $-2^{-6}\Delta=A_4^3+27 A_6^2\equiv A_4+3 \equiv 0\pmod{8}.$ Thus, $a_4\equiv -3\pmod{8},$ and there is some $t\in \Z_2$ so that $a_4=-3t^2$. Moreover, there is a choice of sign of $t$ so that $a_6=2t^3+2^{d-6} s,$ for some $2$-adic unit $s$. If $t\equiv 1\pmod{4}$ (so that $a_6\equiv 2\pmod{8}$), then we find that Tate's algorithm terminates in either step $6$ or $7$. However if $t\equiv 3\pmod{4}$ (so that $a_6\equiv 6\pmod{8}),$ then the substitution $(x,y)\to (4x+t,8y+4x)$ reduces the equation of the curve. 
\begin{equation}\label{Reduced2.2}
y^2+xy=x^3 +\frac{3t-1}{4}x^2+2^{d-12}s
 \end{equation}
This completes the proof of the lemma, as this situation is case (3).
 \end{proof}

Our analysis of Tate's algorithm requires invariants $k, s,t$, and $v$ whenever (\ref{BAD2}) is satisfied with $\alpha_4=0$ and $v_2(\Delta)\geq 8.$ As in the proof above, this implies that $(a_4,a_6)\equiv (1,2)\pmod{4}.$ 
If $d$ is even, we set $v=2^{\frac{d-6}{2}},$ and otherwise we set $v=0.$ Then there is
a unique $t\in \Z_2$ with $t\equiv a_6/2 \pmod{4},$ so that $a_4=-3t^2+2v$. This uses the fact that   $d=8$ if and only if $a_4\equiv 1\pmod{8}.$ A short calculation gives numbers $s$ and $k,$ where $a_6=2t^3-2vt+v^2+ 2^ks,$ with either $s=0,$ or  $s$ is odd and 
$k\geq d-6.$ We find that $k>d-6$ if $d$ is even (so that $v\neq 0$), and $k=d-6$ if $d$ is odd (so that $v=0$).
After the substitution $(x,y)\to (x+t,y+x+v)$, we obtain
\begin{equation}\label{Case2exceptional}
 y^2+2xy+2vy=x^3+(3t-1)x^2+2^k s.
\end{equation} 

As in the previous subsections, we reformulate the algorithm for $p=2$ by identifying its steps with
suitable disjoint possibilities for the Tate data $a_4,a_6, \alpha_4, \alpha_6, d, $ and $k$. Unlike the case where $p\geq 5$, a further cases  arise due to the fact that $2-$minimal models in short Weierstrass form do not always exist for $E$. {\color{black} If $E(a_4,a_6)$ is minimal it will fall into one of cases (1)-(10). Non-minimal curves satisfying condition (1) of Lemma \ref{Minimalp2} will fall into case (11) and iterate through the algorithm again as in the case where $p\geq 5$, while non-minimal curves satisfying conditions (2) or (3) of Lemma \ref{Minimalp2} will fall into extra cases designated (1*) or (2*).
} 
We define $\widehat{\delta}_2(K,\tp)$ to be the proportion of curves which fall into these cases.

\smallskip
\noindent
{\bf Case 1} ({\bf None}). 
This case is Kodaira type $I_0$, which has good reduction at $p=2.$ This case does not occur in a first pass through the algorithm since $16\mid \Delta$.
Therefore, we have $\delta_2'(I_0,1)=0.$

\smallskip
\noindent 
{\bf Case 1*}. This is Kodaira type $I_0$ with $c_2=1,$ where $E$ is not $2$-minimal. The two possibilities for this case are:
\begin{enumerate}[label=(\roman*)]
\item We have $\alpha_4\geq 4$ and $a_6\equiv16\pmod{64}.$
\item We have $(a_4,a_6)\equiv(5,6)\pmod8,$ and $d=12.$
\end{enumerate}
Under the condition (i), we use the model (\ref{Reduced2.1}). After making the substitution $(x,y)\to (4x, 8y+4),$ we obtain
\[ y^2+y=x^3 +A_4x+2^{k-6} s.\]
The new discriminant $\Delta/64=-64A_4^3-27(1+2^{k-2}s)^2$ is odd. The proportion of these curves is $1/2^{10}.$

Under the condition (ii), we use the model (\ref{Reduced2.2}). We have $a_4=-3t^2,$ with $t\equiv 3\pmod{4},$ and $a_6\equiv 2t^3+64s.$ After making the substitution $(x,y)\to (4x+t, 8y+4x),$ we obtain
\[ y^2+xy=x^3 +\frac{3t-1}{4}x^2+s.\]
  The discriminant of this model is $\Delta/64=-27(s\,t^3+16 s^2),$ which is odd. 
We may take any choice of $a_4\equiv 5\pmod{8}$. This choice determines $t$, and therefore determines $a_6\pmod{128}.$ The proportion of curves satisfying this situation is $1/2^{10}.$
Therefore, Tate's algorithm gives $\widehat \delta_2(I_0,1)=1/1024+1/1024=1/512.$

\smallskip
\noindent
{\bf Case 2} ({\bf{None}}). This case  is for Kodaira types $I_{n}.$  There are no short form curves in this case as in Case 2 for $p=3$ (i.e. $2\mid b_2$ because $b_2\equiv \underbar{a}_1^2\pmod 4$). Therefore, we have $\delta_2'(I_{n\geq1},c_2)=0.$

\smallskip
\noindent 
{\bf Case 2*} ($\boldsymbol{(a_4,a_6)\equiv(5,6)\pmod8, d>12}$). This case is for the Kodaira type $I_{n\geq 1}$, with $n=d-12$, where $E$ is not $2$-minimal. 
We use model (\ref{Reduced2.2}). We have $a_4=-3t^2$ and $a_6=2t^3+2^{d-6}s,$  where $t\equiv 3\pmod{4}.$  After making the substitution $(x,y)\to (4x+t, 8y+4x),$ we obtain
\[ y^2+xy=x^3 +\frac{3t-1}{4}x^2+2^{d-12}s.\]
 The discriminant is $\Delta/64=-27(2^{d-12}s\,t^3+2^{2d-20} s^2).$ Tate's algorithm gives $c_2\in \{1, 2, n\},$ depending on the polynomial $T^2+T+\frac{3t-1}{4}$ modulo $2$. We have  $c_2:=n$ if it has roots modulo $2$ (i.e if $t\equiv 3\pmod{8}$). If it does not have roots modulo $2$ (i.e. if $t\equiv 5\pmod{8}$), then $c_2:=1$ (resp. $c_2:=2$) if $n$ is odd (resp.  even).
 
  Therefore to compute the proportions, for any  $n$ we may take any $a_4\equiv 5\pmod{8},$ which determines $t\in \Z_2,$ with $t\equiv 3 \pmod{4}.$ This then determines $a_6\equiv 2t^3+2^{n+6}\pmod{2^{n+7}}.$ Using $\varepsilon(n)$ as in Case 2 of Subsection~\ref{TateP5},
the algorithm gives that 
$\delta'_2(I_{1},1)=1/2^{11},$ $\delta'_2(I_{2},2)=1/2^{12},$ and $\delta'_2(I_{n\geq3},n)=\delta'_2(I_{n\geq 3},\varepsilon(n))= 1/2^{11+n}.$

\smallskip
\noindent
{\bf Case 3} ($\boldsymbol{ (a_4,a_6)\in\{ (0,2),(0,3), (1,0), (1,1), (2,2),(2,3), (3,2), (3,3)\pmod 4\}}$).
  This case is Kodaira type $II,$ where $c_2=1.$ A
 brute force analysis shows that these cases  correspond only to the indicated congruence conditions.
 As these account for $8$ out of the $16$ possible pairs $(a_4,a_6)$ modulo $4,$ we obtain
$\delta_2'(II,1)=1/2.$

\smallskip
\noindent
{\bf Case 4} ($\boldsymbol{ (a_4,a_6)\in\{ (1,3), (2,1), (2,0), (3,0)\pmod4\}} $).
  This case is Kodaira type $III,$ where $c_2=2.$ 
  A brute force analysis shows that these cases  correspond only to the indicated congruence conditions.
  As these account for $4$ out of the $16$ possible pairs $(a_4,a_6)$ modulo $4,$ we obtain
$\delta_2'(III,2)=1/4.$

\smallskip
\noindent
{\bf Case 5} ($\boldsymbol{(a_4,a_6)\in\{ (0,1), (3,1)\pmod4\}}$). 
 This case is for Kodaira type $IV$, where $c_2\in \{1, 3\},$ depending on the parity of $(a_6+a_4t+t^3-v^2)/4$. The algorithm gives $c_2:=3$  (resp. $c_2:=1$) if this number is even (resp.  odd).  
By brute force, we find that $c_2=1$ for 
\[(a_4,a_6)\pmod{8}\in\{ (0,5), (3,1), (4,5),(7,5)\pmod 8\},\]
and $c_2=3$ for 
\[(a_4,a_6)\pmod{8}\in\{ (0,1), (3,5), (4,1),(7,1)\pmod 8\}.\]
Therefore, we obtain $\delta_2'(IV,1)=1/16,$ and $\delta_2'(IV,3)=1/16.$

\smallskip
\noindent
{\bf Case 6}. 
This case is for Kodaira type $I_0^*$, where $c_2\in \{1,2\}.$ The following two possibilities for this case are:
\begin{enumerate}[label=(\roman*)]
\item We have $\alpha_4\geq 2$ and $a_6\equiv8,12\pmod{16}.$
\item We have $(a_4,a_6)\equiv(1,2)\pmod4,$ and $k=3,$ where $k$ is defined as in (\ref{Case2exceptional}).
\end{enumerate}
Assuming (i), there are numbers $s$ and $v$ for which $a_6=8s+v^2$, where $s$ is odd and $v=0$ (resp. $v=2$) when $a_6\equiv 8\pmod{16}$ (resp. when $a_6\equiv 12\pmod{16}$). After making the substitution $(x,y)\to(x,y+v),$
we obtain $y^2+2vy=x^3+a_4x+8s.$
Using this model, we define (as in (\ref{Def_P}))
\[P(T)=T^3+\frac{1}{4}a_4 T +s.\]
If $a_4\equiv 4\pmod{8},$ then $P(T)$ is irreducible modulo $2$ and so $c_2:=1.$  If $a_4\equiv 0 \pmod{8},$ then $P(T)$ has one root modulo $2$, and so $c_2:=2$. Therefore the contributions from condition (i) to $\delta_2'(I_0^*,1)$ and $\delta_2'(I_0^*,2)$ are both $1/64.$

For curves satisfying (ii), we use (\ref{Case2exceptional}) to define
\[P(T)=T^3+\frac12(3t-1)T^2+s,\]
and Tate's algorithm gives $c_2:=1$ when $t\equiv 1\pmod 4,$ and $c_2:=2$ when $t\equiv 3\pmod{4}.$ Since $k=3,$ we have $v=0$ or $2$. To calculate the proportion of curves satisfying (ii), we may pick any  $a_4\equiv 1\pmod{4}$. Then we have $v=0$ when $a_4\equiv 1\pmod{8},$ and $v=2$ otherwise. The choice of $a_4$ and $c_2$ fixes $t$ uniquely, which then determines $a_6\equiv 2t^3-2vt+v^2+8\pmod{16}.$ Thus, for each choice of $c_2$, condition (ii) contributes a proportion of $1/64$ to $\delta_2'(I_0^*,c_2).$
  Therefore Tate's algorithm gives $\delta_2'(I_0^*,1)=\delta_2'(I_0^*,2)=1/64+1/64=1/32.$

\smallskip
\noindent
{\bf Case 7}.  
This case is for Kodaira type $I^*_{n\geq 1 },$ where $c_2\in \{2, 4\}.$  The only possibilities for $(\alpha_4,\alpha_6)$ are:
\begin{enumerate}[label=(\roman*)]
\item We have $\alpha_4= 2$ and $a_6\equiv0,4\pmod{16}.$
\item We have $a_4\equiv1\pmod4,$ $a_6\equiv2\pmod8,$ and $k\geq4,$ where $k$ is defined in (\ref{Case2exceptional}).
\end{enumerate}
Instead of proceeding as in the previous cases, we determine the conditions which result in any given choice of $n$ and $c_2\in \{2, 4\}.$  We are essentially working backwards through iterations of the sub-procedure in Step 7 of the algorithm. In particular, $n$
is the number of iterations required. As illustrated in the two previous subsections, this step of the algorithm makes use of two auxiliary polynomials,  $P(T)$ and $R(Y)$, which are defined by  (\ref{Def_P}) and (\ref{Def_R}), from a long model
 (\ref{LongModel}) with  $2\mid \underbar a_1$, $\underbar a_2\equiv 2\pmod{4},$ $4\mid \underbar a_3$, $8\mid \underbar a_4$ and $16\mid \underbar a_6.$ 
 
Suppose that $n=2a+1$ is odd. Then the sub-procedure finds a model for which $$2^{-2a}R(2^{a}Y)\equiv Y^2+Y\ {\text {\rm or }}\ Y^2+Y+1\pmod 2,
$$
and also satisfies
$$2^{-2a}P(2^{a}T)\equiv \begin{cases} T^2\pmod{2} \ \ \ \ &{\text {\rm if $a>0$,}}\\
T^3+T^2\pmod{2} \ \ \ \ &{\text {\rm if $a=0$}}.
\end{cases}
$$ 
The point here is the $2^{-2a}P(2^aT)\pmod 2$ has a double root at $T=0.$

These conditions are equivalent to the existence of $A,B, C\in \Z,$ with $A$ odd, so that $\underbar a_3=2^{a+2}A,$ $\underbar a_4=2^{a+3}B,$ and $\underbar a_6=2^{2a+4}C$. If $C$ is even, then we see that $R(Y)$ factors over $\Z_2$, and so the algorithm gives $c_2:=4.$ Otherwise, we have $c_2:=2.$
The substitutions $y\to y+2x$ and $y\to y+2^{a+2}$ do not alter any of the required conditions on $P(T)$ and $R(Y)$. Therefore, we may assume without loss of generality $\underbar a_1=2u$ with $u\in \{0,1\},$ and $A=1$. Similarly, the substitution $x\to x+2^{a+2}$ does not alter the required conditions, and so we may assume that $\underbar a_2=3t_0-u^2$ for some $t_0\equiv 1\pmod{4}$ (if $u=1$) or $t_0\equiv 2\pmod{4}$ (if $u=0$), and $0<t_0<2^{a+2}.$ Then the substitution $(x,y)\to (x-t_0, y-ux-2^{a+2})$ returns the equation of the curve to Weierstrass short form,
$y^2=x^3+a_4x+a_6,$ where we see that $a_4=-3t_0^2+2^{a+2}u+2^{a+2}B$, and $a_6=-t_0^3-t_0 a_4 +2^{2a+2}+2^{2a+4}C.$ We therefore have $2$ choices for $u$ and $2^{a}$ choices for $t$ depending on $u$. This determines $a_4\pmod{2^{a+3}}.$ Together with $c_2,$ this determines $a_6\pmod{2^{2a+5}}.$ Thus for a fixed odd $n$ and choice of $c_2$, we see that $\delta_2'(I_{n\geq 1},c_2)=2^{a+1}\cdot {1}/{2^{a+3}}\cdot {1}/{2^{2a+5}}={1}/{2^{n+6}}.$

Now suppose that $n=2a$ is even. Then the sub-procedure finds a model for which
$$2^{-2a}P(2^{a}T) \equiv T^2+T\ {\text {\rm or }}\ T^2+T+1\pmod 2,
$$
 and satisfies
$2^{-2a+2}R(2^{a-1}Y)\equiv Y^2\pmod{2}.$ This is equivalent to the existence of integers $A,B,$ and $C,$ with $B$ odd, so that $\underbar a_3=2^{a+2}A,$ $\underbar a_4=2^{a+2}B,$ and $\underbar a_6=2^{2a+3}C$. If $C$ is even, then we see that $R(T)$ factors over $\Z_2$, and the algorithm gives $c_2:=4.$ Otherwise, we have $c_2:=2.$
As before, we note that the substitutions $y\to y+2x$ and $y\to y+2^{a+2}$ do not alter any of the required conditions. Therefore, we may assume without loss of generality $\underbar a_1=2u$ with $u\in \{0,1\},$ and $A=0$. Similarly, the substitution $x\to x+2^{a+2}$ does not alter the conditions, so we may assume that $\underbar a_2=3t_0-u^2$ for some $t_0\equiv 1,2\pmod{4}$ (depending on $u$), and $0<t_0<2^{a+2}.$ After making the substitution $(x,y)\to (x-t_0, y-ux),$ we obtain the short Weierstrass model,
$y^2=x^3+a_4x+a_6,$ where $a_4=-3t_0^2+2^{a+2}B$, and $a_6=-t_0^3-t_0a_4 +2^{2a+3}C.$ We therefore have $2$ choices for $u$ and $2^{a}$ choices for $t_0$ depending on $u$. This determines $a_4\pmod{2^{a+3}}.$ Together with the choice of $c_2,$ this determines $a_6\pmod{2^{2a+4}}.$ Hence, for a fixed odd $n$ and choice of $c_2$, we see that $\delta_2'(I_{n\geq 1},c_2)=2^{a+1}\cdot {1}/{2^{a+3}}\cdot {1}/{2^{2a+4}}={1}/{2^{n+6}}.$ 

A short calculation shows that the case $u=0$ implies that $(a_4,a_6)$ satisfies (i), where as $u=1$ implies that $(a_4,a_6)$ satisfies (ii).
In summary, for each $n\geq 1$ we have $\delta_2'(I_{n\geq 1},2)=\delta_2'(I_{n\geq 1},4)=1/2^{n+6}.$

\smallskip
\noindent
{\bf Case 8}.
This case is for Kodaira type $IV^*$, where $c_2\in \{1, 3\}$. The following two possibilities for this case are:
\begin{enumerate}[label=(\roman*)]
\item We have $\alpha_4\geq 3$ and $a_6\equiv4\pmod{16}.$
\item We have $(a_4,a_6)\equiv(1,6)\pmod8,$ and $k\geq4.$ Note that $k$ is from model (\ref{Case2exceptional}).

\end{enumerate}
For (i), the algorithm implies that $c_2:=1$ if $a_6\equiv 20 \pmod{32},$ and $c_2:=3$  when $a_6\equiv 4 \pmod{32}.$ 
Therefore, (i) contributes a proportion of 1/256 for both $c_2=1$ and $c_2=3.$
For condition (ii), the algorithm implies that $c_2:=1$ if $k=4$ (resp. $c_2:=3$ if $k\geq 5$). In this situation, we have $t\equiv 3\pmod{4},$ and $v=2$. The condition that $k\geq 4$ then implies that  $(a_4,a_6)\equiv (1,6)$ or $(9,14) \pmod{16}.$ Half of each set of possible pairs $(a_4,a_6)$ will correspond to each possible $c_2,$ thereby contributing another proportion of 1/256, and so we obtain
$\delta_2'(IV^*,1)=\delta_2'(IV^*,3)=1/128.$

\smallskip
\noindent
{\bf Case 9}. 
This is for Kodaira type $III^*,$ where $c_2=2$. The following two possibilities for this case are:
\begin{enumerate}[label=(\roman*)]
\item We have $\alpha_4=3$ and $\alpha_6\geq4.$
\item We have $(a_4,a_6)\equiv(5,6)\pmod8,$ and $k=4.$ Note that $k$ is from model (\ref{Case2exceptional}).
\end{enumerate}
For condition (i), we have $(a_4,a_6)\equiv (8,0)\pmod{16},$ and so the proportion of curves in this case is $1/256.$ For condition (ii), we have $a_4=-3t^2,$ and $a_6\equiv 16-a_4t-t^3\pmod{32},$ which implies that $t\equiv 3\pmod 4.$
Therefore, by brute force we find that $(a_4,a_6)\pmod{32}\in \{ (5,6), (13,30), (21,22), (29,14)\pmod{32}\},$ representing a proportion 1/256, and so we obtain $\delta_2'(III^*,2)=1/256+1/256=1/128.$

\smallskip
\noindent
{\bf Case 10}. 
This case is Kodaira type $II^*,$ where $c_2=1.$  The following two possibilities for this case are:
\begin{enumerate}[label=(\roman*)]
\item We have $\alpha_4\geq 4$ and $a_6\equiv32,48\pmod{64}.$
\item We have $(a_4,a_6)\equiv(5,6)\pmod8,$ and $k=5.$ Note that $k$ is from model (\ref{Case2exceptional}).
\end{enumerate}
Clearly, the proportion of curves satisfying (i) is $1/512.$
For (ii), we note that
$v=0$, and that $t$ is uniquely determined by $a_4\equiv 5\pmod{8},$ which determines $a_6\equiv -a_4t-t^3 +32 \pmod{64}.$ 
Therefore, the proportion of curves in this case is also 1/512, and so
 $\delta_2'(II^*,1)=1/512+1/512=1/256.$

\smallskip
\noindent
{\bf Case 11} ($\boldsymbol{\alpha_4\geq 4 \textbf{ and } \alpha_6\geq 6}$). As the model is not minimal,  the algorithm replaces $a_4$ and $a_6$ with $a_4/16$ and $a_6/64$ respectively. One repeats these substitutions until one obtains a model which is one of the ten cases above.

\section{Proofs}\label{Proofs}

\subsection{Tamagawa Numbers and the proof of Theorems~\ref{Tam_m} and \ref{AverageTamagawa}}\label{AverageTamagawaProof}
 
Using the results from the previous section, we now compute each  $\delta_p(n)$, the proportion of curves $E(a_4,a_6)$ whose $p$-minimal models have $c_p=\tp.$

\begin{lemma}\label{Four} If $p$ is prime and $\tp\geq 1$, then the following are true.
\begin{enumerate}
\item For $p=2,$ we have $\delta_2(1)=241/396,$ $\delta_2(2)=7495/24552,$ $\delta_2(3)=1153/16368,$ and
$\delta_2(4)=171/10912.$ Moreover, if $n\geq 5$, then
$$\delta_2(\tp)=\displaystyle{\frac{1}{2^{\tp+1}\cdot 1023}}.
$$

\item For $p=3$, we have $\delta_3(1)=1924625/2125728,$ $\delta_3(2)=510641/6377184,$ $\delta_3(3)=7594/597861,$ and
$\delta_3(4)=1193/652212.$ Moreover, if $n\geq 5$, then
$$\delta_3(\tp)=\displaystyle{\frac{1}{3^{\tp+1}\cdot 29524}}.
$$

\item If $p\geq 5$ is prime, then we have

$$
\delta_p(\tp)=
\begin{cases} {\color{black}
\displaystyle\ 1-\frac{p(6p^7+9p^6+9p^5+7p^4+8p^3+7p^2+9p+6)}{6(p+1)^2(p^8+p^6+p^4+p^2+1)}} \ \ \ \ &{\text if }\ \tp=1,\\[14pt]
{\color{black}\displaystyle \ \ \ \frac{p(2p^7+2p^6+p^5+p^4+2p^3+p^2+2p+2)}{2(p+1)^2(p^8+p^6+p^4+p^2+1)}} \ \ \ \  &{\text if}\  \tp=2,\\[14pt]
\displaystyle\ \ \ \frac{p^2(p^4+1)}{2(p+1)(p^8+p^6+p^4+p^2+1)} \ \ \ \ &{\text if}\ \tp=3,\\[14pt]
\displaystyle\ \ \ \frac{p^3(3p^2-{2p}-1)}{6(p+1)(p^8+p^6+p^4+p^2+1)} \ \ \ \ &{\text if}\  \tp=4,\\[14pt]
\displaystyle \ \ \ \frac{p^{10}-2p^9+p^8}{2p^{\tp}(p^{10}-1)} \ \ \ \ &{\text if }\ \tp\geq 5.\\
\end{cases}
$$

\end{enumerate}
\end{lemma}
\begin{proof} 
We first prove (3), the formulas for $\delta_p(\tp),$ where $p\geq 5$ is prime. In the previous section, we computed the numbers
 $\delta'_p(K,\tp),$ the proportion of $p$-minimal short Weierstrass models with Kodaira type $K$ and Tamagawa number $c_p=\tp.$ 
 We determine $\delta_p(\tp)$ from the $\delta'_p(K,\tp)$ by keeping track of the distribution of all short Weierstrass models
 onto the $p$-minimal models as dictated by Tate's algorithm.  
 Thanks to Lemma~\ref{Minimalp5}, we only need to consider the iterations
 of substitution in case eleven, which takes into account the divisibility of
 $a_4$ (resp. $a_6$) by powers of $p^4$ (resp. $p^6$).  Moreover, $1/p^{10^\tp}$ represents the proportion of curves that pass through at least $n$ additional iterations before satisfying one of the first ten cases.
 Therefore, we obtain the formula
 \begin{equation}\label{IterationFormula}
 \delta_p(\tp):=\sum_{K} \delta'_p(K,\tp)\cdot \left(1+\frac{1}{p^{10}}+\frac{1}{p^{20}}+\dots\right)=
 \frac{p^{10}}{p^{10}-1}
 \sum_{K}\delta'_p(K,\tp).
 \end{equation}
 The formulas are obtained using the entries in Table~\ref{pgeq5} in the Appendix.
For example, if $\tp=1,$ then we have
\begin{displaymath}
\begin{split}
\delta_p(1)&=\frac{p^{10}}{p^{10}-1}\cdot \left(\frac{p-1}{p}+\frac{(p-1)^2}{p^3}+\frac{(p-1)}{p^3}+\frac{(p-1)}{2p^5}+\frac{(p^2-1)}{3p^7}+\frac{(p-1)}{2p^8}+\frac{(p-1)}{p^{10}}\right)
 + \frac{p^{10}}{p^{10}-1}{\color{black}\sum_{\substack{{\color{red}n}=3\\{\color{red}n}~\mathrm{odd}}}^{\infty}}\frac{(p-1)^2}{2p^{n+2}}\\
&={\color{black}1-\frac{p(6p^7+9p^6+9p^5+7p^4+8p^3+7p^2+9p+6)}{6(p+1)^2(p^8+p^6+p^4+p^2+1)}}.
\end{split}
\end{displaymath}
The infinite sum on ${\color{red}n}\geq 3$ corresponds to the Kodaira types $I_{{\color{red}n}\geq 3},$ where ${\color{red}n}$ is odd.

We now turn to the proof of (2).
Curves satisfying condition (1) or (2) of Lemma \ref{Minimalp3} are not included in the proportions $\delta'_p(K,\tp).$ We replace the curves satisfying (2) with the models given in (\ref{Reduced3}). Since $t$ is a $3$-adic unit, these curves cannot be transformed into short Weierstrass form without either introducing a denominator of $3$ or increasing the discriminant. After a second pass through the algorithm, we find that these curves terminate in either case $1$ or $2.$ We designated these situations in the previous subsection with an asterisk, and we denote the proportion of such curves satisfying (2) with minimal model with Kodaira type $K$ and Tamagawa number $c_3=\tp$ by $\widehat\delta_3(K,\tp). $  As above, we have that $1/3^{10}$ is the proportion of curves satisfying condition (1), and which then pass through the algorithm again. More generally, $1/3^{10^m}$ represents the proportion of curves that pass through at least $m$ additional iterations before satisfying one of the first ten cases, or condition (2) of Lemma~\ref{Minimalp3}. We may therefore use a geometric series to count these curves. This leads to the formula
 \begin{equation}\label{IterationFormula}
 \delta_3(\tp):=\sum_{K} (\delta'_3(K,\tp)+\widehat\delta_3(K,\tp))\cdot \left(1+\frac{1}{3^{10}}+\frac{1}{3^{20}}+\dots\right)=
\frac{3^{10}}{3^{10}-1}\cdot \sum_{K}(\delta'_3(K,\tp)+\widehat\delta_3(K,\tp)).
\end{equation}
By brute force calculation using the entries listed in Tables~\ref{phattwothree} and ~\ref{ptwothree} in the Appendix, we obtain (2).

Finally, we prove (1). Curves satisfying condition (1), (2), or (3) of Lemma \ref{Minimalp2} are not included in the proportions $\delta'_p(K,\tp).$ We replace the curves satisfying (2) with the models given in (\ref{Reduced2.1}), and the curves satisfying (3) with the models given in (\ref{Reduced2.2}). Since either the coefficient of $y$ or of $xy$ in the reduced model is odd, these curves cannot be transformed into short Weierstrass form without either introducing a denominator of $2$ or increasing the discriminant. After a second pass through the algorithm, we find that these curves terminate in either case $1$ or $2.$ We designated these situations in the previous subsection with an asterisk, and we denote the proportion of such curves satisfying (2) with minimal model with Kodaira type $K$ and Tamagawa number $c_2=\tp$ by $\widehat\delta_2(K,\tp). $   Moreover as above, $1/2^{10}$  is the proportion of curves that satisfy condition (1), and pass through the algorithm again. More generally, $1/2^{10^m}$ is the proportion of curves that pass through at least $m$ additional iterations of the algorithm before satisfying one of the first ten cases, or satisfies (2) or (3) of
Lemma~\ref{Minimalp2}. We may therefore use a geometric series to count these curves. 
Therefore, it follows that the analog of (\ref{IterationFormula}), again using Tables~\ref{phattwothree} and ~\ref{ptwothree}, is
$$
\delta_2(\tp):={\color{black}\frac{2^{10}}{2^{10}-1} \sum_{{\text {Type $K$}}}(\delta_2'(K,\tp)+\widehat \delta_2(K,\tp))}
$$

\end{proof}

\begin{proof}[Proof of Theorem~\ref{Tam_m}] By Lemma~\ref{Four} and the Chinese Remainder Theorem, we have
$$
P_{\Tam}(1):=\lim_{X\rightarrow +\infty}\frac{\NN_{1}(X)}{\NN(X)}=\prod_{p \ {\text {\rm prime}}} \delta_p(1)={\color{black}0.5053\dots.}
$$
More generally, by multiplicativity, we formally find that
$$
L_{\Tam}(s):=\sum_{m=1}^{\infty}\frac{P_{\Tam}(m)}{m^s}=\prod_{p\ prime} \left(\frac{\delta_p(1)}{1^s}+\frac{\delta_p(2)}{2^s}+\frac{\delta_p(3)}{3^s}+\dots\right).
$$
To complete the proof it suffices to verify the convergence of the Dirichlet coefficients defined by this infinite product.
To this end, we note that Lemma~\ref{Four}~(3) establishes,
for primes $p\geq 5$, that
$1-\frac{1}{p^2} < \delta_p(1)< 1.$
Therefore, convergence follows by comparison with $1/\zeta(2)=\prod_p\left(1-\frac{1}{p^2}\right)=6/\pi^2.$
\end{proof}

\begin{proof}[Proof of Theorem~\ref{AverageTamagawa}]
Using Lemma~\ref{Four}, we find that the  ``average value'' of $\Tam(E(a_4,a_6))$ is
\begin{equation}\label{EulerProduct}
L_{\Tam}(-1)=\sum_{m=1}^{\infty} P_{\Tam}(m)m=\prod_{p\ prime}\left(\delta_p(1)+2\delta_p(2)+3\delta_p(3)+\dots\right)
\end{equation}
provided that this expression is convergent. 
To this end, we apply Lemma~\ref{Four}.
For  $p\geq 5,$ Lemma~\ref{Four} gives $\delta_p(1)=1-\frac{1}{p^2}+o(1/p^2)$, and for $n\geq 2$ gives
$0<\delta_p(\tp)\tp<\frac{\tp}{p^\tp}.$
Therefore, since $\sum_{\tp=1}^{\infty}\frac{\tp}{p^\tp}=p/(p-1)^2$ we have
$$
\sum_{\tp=1}^{\infty}\delta_p(\tp)\tp= 1 -\frac{1}{p^2}+O\left(\frac{1}{p^2}\right).
$$
Similarly, we have convergence for $p\in \{2, 3\},$ and we have
$\sum_{\tp\geq 1}\delta_2(\tp)\tp ={\color{black}1.4941\dots}$ and
$\sum_{\tp\geq 1}\delta_3(\tp)\tp ={\color{black}1.1109\dots}.$
The convergence of (\ref{EulerProduct})  follows by multiplicativity, and with
a computer one finds
$L_{\Tam}(-1)={\color{black}1.8193\dots}.$
\end{proof}

\subsection{Proof of Lemma~\ref{Convenient}}\label{ConvenientLemma}
We expand on an example of Buhler, Gross and Zagier \cite{BGZ}, which is based on a method of Tate
(for example, see \cite{Silverman88}).
For $E=E(a_4,a_6)$, we define 
\begin{equation}\label{F_E}
F_E(x):=\frac{1}{2}\log| x|_{\infty}+\frac{1}{8}\sum_{n=0}^{\infty}\frac{\log |z_n|_{\infty}}{4^n},
\end{equation}
where $|\cdot|_\infty$ is the usual archimedean valuation of $\R,$ and where for $n\geq 0$ we let $x_0:=x$ and
\begin{align}\label{sequence}
  z_n:=1-\frac{2a_4}{x_n^2}-\frac{8a_6}{x_n^3}+\frac{a_4^2}{x_n^4} \quad   \mathrm{and} \quad x_{n+1}:=\frac{x_n^{4}-2a_4x_n^2-8a_6x_n+a_4^2}{4(x_n^3+a_4x_n+a_6)}.
\end{align}
For $P\in E(\Q)$, if $x:=x(P),$  then we have $x_{n}=x(2^{n}P).$ 
Moreover, as $n\rightarrow +\infty$, 
$F_E(x_n)-\log|x_n|_{\infty}/2$ converges. This definition of $F_E(x)$  differs from that used in \cite{BGZ} by a factor of $2$ due to our normalization of  $\widehat{h}(P).$

Let $\mathcal{O}\neq P\in E(\Q)$ be a rational point. The canonical height $\widehat{h}(P)$ can be computed as a sum of local heights (for example, see Ch. VI of \cite{SilvermanAdvanced})
$$ \widehat{h}(P)=\sum_{v}\widehat{h}_{v}(P),
$$
where the sum is over all the places of $\Q$ (including $\infty$).
Since $E$ is Tamagawa trivial, 
every point of $E(\Q)$ reduces to a non-singular point modulo any prime $p$. Therefore, equation (26) of \cite{Silverman88} (with a sign error fixed)
  shows that the contribution from each finite place $v$ is exactly 
  \[\widehat{h}_{v}(P)=\max\{0, 1/2 \log|x(P)|_v\},\] with no additional contribution from the primes of bad reduction.
  The archimedean component is 
\[\widehat{h}_{v}(P) = F_E(x(P)),\]
assuming \eqref{F_E} converges (see Theorem 1.2 of \cite{Silverman88}). 

Now we turn to the two claims in the lemma.
Let $D:=[\alpha,\infty)$ if $E(\R)$ has one connected component, and $D:=[\gamma,\beta]\cup[\alpha,\infty)$ if $E(\R)$ has two connected components. In both cases we have $ 2a_4x^2+8a_6x-a_4^2<0$ for all $x\in D$, which in turn implies that $z_n>1$.  In order to show that \eqref{F_E} converges, we must show that 
 the $x_n$ are bounded away from $0$ for $n\geq 1$. 
 We first define $\varepsilon:=\inf_{x\in D} (-2a_4x^2-8a_6x+a_4^2)>0.$  We then set 
 $$s_1:=\sup_{\substack{x\in D\\ |x|\geq1}} \frac{x^3+a_4x+a_6}{x^3} \ \ \ \ \  {\text {\rm and}}\ \ \ \ \ s_2:=\sup_{\sup_{\substack{x\in D\\ x<1}}} x^3+a_4x+a_6.
 $$
 We find that if $x_n\geq 1$, then $x_{n+1}\geq \frac{x_n^4+\varepsilon}{x^3 s_1}\geq \frac{1}{s_1}.$ If $\alpha<x_n<1,$ then $x_{n+1}\geq \frac{x_n^4+\varepsilon}{s_2}\geq \frac{\varepsilon}{s_2}.$ Since the $x_{n+1}$ is always bounded away from $0$, we have an upper bound on $z_{n+1}.$ Thus \eqref{F_E} converges, as long as our initial $x$ is not $0$. If $x=0$, then we follow \cite{BGZ}, combining the $\log|x|$ term and the first term of the sum to obtain
 \begin{align}\label{F_E*}
 F_E(x)=\frac{1}{8}\log|x^4-2a_4x^2-8a_6x+a_4^2|_{\infty}+\frac{1}{8}\sum_{n=1}^{\infty}\frac{\log |z_n|_{\infty}}{4^n},
 \end{align}
 which can be evaluated even when $x=0.$
Summing the local heights, we obtain
\begin{align}\label{keyid}
    \widehat{h}(P)\, =\,  \frac{1}{2}\log{|C^2|_\infty}+F_{E}(x(P)) \, = \,
    \frac{1}{2}h_W(P)+F_{E}(x(P))-\frac{1}{2}\log \max(|x(P)|,1)
\end{align}
where $x(P)=A/C^2$ in lowest terms.

If $|x(P)|\geq 1,$ then
 $h_W(P)=\log|A|_\infty$, and so 
\begin{align*}
    \widehat{h}(P) - \frac{1}{2}h_W(P) =\frac{1}{8}\sum_{n=0}^{\infty}\frac{\log |z_n|_{\infty}}{4^n}>0.
\end{align*}
If $E(\R)$ has one connected component, then Cardano's formula confirms that $\alpha>1,$ Thereby proving case (1) since $x(P)>\alpha>1.$
To complete the proof, we need only consider case (2), with $|x(P)|<1$. In this case, $h_W(P)=\log|C^2|_\infty$, and so $ \widehat{h}(P) - \frac{1}{2}h_W(P)=F_E(x(P))$. However, by hypothesis, we have that $F_E(x)\geq0$ for all $x\in D$.

\subsection{Proof of Corollary~\ref{MainCorollary}}\label{CorollaryProof}

Before we prove Corollary~\ref{MainCorollary}, we begin with an auxiliary lemma
which establishes that the vast proportion of curves $E(a_4,a_6)$ with bounded height are already minimal models.
Namely, we let
\begin{equation}\label{Nmin}
\NN_{\min}(X):=\# \{ E(a_4,a_6) \ {\text {\rm a minimal model}} \ : \  \height(E(a_4,a_6))\leq X\}.
\end{equation}

\begin{lemma}\label{Minimal} As $X\rightarrow +\infty$, we have
$$
\rho:=\lim_{X\rightarrow +\infty} \frac{\NN_{\min}(X)}{\NN(X)}=\frac{21342914775}{228811\pi^{10}}=0.9960\dots.
$$
\end{lemma}
\begin{proof}
For primes $p\geq 5,$ Lemma \ref{Minimalp5}, shows that the only  short Weierstrass models which are not $p$-minimal have $v_p(a_4)\geq 4$ and $v_p(a_6)\geq 6$. Therefore, the multiplicative contribution to $\rho$ for such primes is $1-1/p^{10}.$ 
Similarly, for $p=2$ (resp. $p=3$), Lemma~\ref{Minimalp2} (resp. Lemma~\ref{Minimalp3}) determines those short Weierstrass models which are $2$-minimal (resp. $3$-minimal).
Using the tables in the Appendix, we find that the proportion of $2$-minimal curves is $\sum_{\mathrm{Type}~K}\delta_2'(K,\tp)=255/256=1-1/2^8$ (resp. $3$-minimal curves is $\sum_{K}\delta_3'(K,\tp)=19682/19683=1-1/3^9$). 
The formula for $\rho$ follows by multiplicativity and the fact that $\zeta(10)=\prod_{p}\left(1-\frac{1}{p^{10}}\right)^{-1}=\pi^{10}/93555.$
\end{proof}

\begin{proof}[Proof of Corollary~\ref{MainCorollary}]
Thanks to Theorem~\ref{Tam_m}, we find that
\begin{equation}\label{formula}
\lim_{X\rightarrow +\infty}\frac{\NN_c(X)}{\NN(X)}=\kappa\cdot P_{\Tam}(1)=\kappa  \prod_{p\ prime} \delta_p(1).
\end{equation}
where
$\kappa$ is the proportion of 
$E(a_4, a_6)$  that are  minimal models that also satisfy one of the following two conditions.

  \begin{enumerate}
 \item We have that $E(\R)$ has one connected component, and 
  $$ a_4\leq 0 \ \ \ {\text {\rm and}}\ \ \ (\alpha,\infty)\subset\left\{x\in \R:2a_4x^2+8a_6x-a_4^2<0\right\},$$ where $\alpha$ is the real root of $x^3+a_4x+a_6$. 
  
 \item We have that
 $E(\R)$ has two connected components, and 
  $$a_4\leq 0 \ \ \ {\text {\rm and}}\ \ \ (\gamma,\beta)\cup(\alpha,\infty)\subset\left\{x\in \R:2a_4x^2+8a_6x-a_4^2<0\right\},$$ where $\gamma<\beta<\alpha$ are the  real roots of $x^3+a_4x+a_6$.
 \end{enumerate}
 Therefore, we have that $\kappa:=\rho\cdot (\kappa_1+\kappa_2)$, where $\rho$ is given in Lemma~\ref{Minimal}, and
  $\kappa_1$ (resp. $\kappa_2$) denotes
 the proportion of $E=E(a_4, a_6)$ with $\height(E)\leq X$  that satisfy condition (1) (resp. (2)).
 
It is convenient to first reformulate these two cases in terms of models over $\R$ given by a single parameter $T$.
To this end, we make use of the change of variable
 \begin{equation}\label{cov}
 (x,y)\to (\sqrt{|a_4|}x, |a_4|^{3/4}y).
 \end{equation}
 By letting  $T:=a_6/|a_4|^{3/2},$ we then obtain
  \[ y^2 =x^3-x+T.
 \]
 If we set $F(x,T)=x^3-x+T$ and $G(x,T)= -2x^2+8Tx-1$, then both (1) and (2) are reformulated
 as 
 \begin{equation}\label{ConditionT}
  G(x,T) <0 \ \ \text{for all} \ \ x\in \R \ \ \text{such that} \ \  F(x,T)>0.
  \end{equation}
  The convenient curves with $a_4=0$ have density 0 
   as $X\rightarrow +\infty.$ Therefore, it suffices to consider (\ref{ConditionT}).
 
It is straightforward to determine when (\ref{ConditionT}) holds using the discriminant of $F(x,T)$.
Indeed, the discriminant is positive (resp. negative) when the curve has 2 real components (resp. 1 real component). Hence, the two cases are determined by the location of
 $T$ in $\R$, with respect to the points satisfying one of the following possibilities:
\begin{itemize}
\item the discriminant of $F(x,T)$ (with respect to $x$) is $0$ 
\item the discriminant of $G(x,T)$ (with respect to $x$) is $0$
\item points $T$ where $F(x,T)$ and $G(x,T)$ share a root. 
\end{itemize}
These conditions are dictated by the common zeros of  $F(x,T)$ and $G(x,T)$. 

The discriminant of $F(x,T)$ with respect to $x$ is $4-27T^2$, which is zero when $T=\pm \frac{2}{\sqrt{27}}.$ The discriminant of $G(x,T)$ with respect to $x$ is $64T^2-8$, which is zero when $T=\pm \frac{1}{\sqrt{8}}.$
To determine when $F$ and $G$ share a root, set 
$$r_\pm(T)=\frac{-8T\pm \sqrt{64T^2-8}}{-4}=2T \mp \sqrt{4T^2-\tfrac12},$$
 which are the two roots in $x$ of $G(x,T).$
Then a straightforward calculation reveals that $F(x,T)$ and $G(x,T)$ share a root in $x$ if and only if $T$ is a root of the polynomial
\[
F(r_+(T))\cdot F(r_-(T)) \ = \ 64T^4-17T^2+\tfrac98 \ = \ \tfrac18 (8T^2-1)(8T+3)(8T-3).
\]
Hence, we have the two additional critical values $T=\pm \frac38$.
By calculating the functions $F(x,T)$ and $G(x,T)$ for $T$ in the various intervals between these critical values, we see that \eqref{ConditionT} is satisfied only when 
$T\in \left(-\infty, -2/\sqrt{27}\right)$ or $T\in \left(-1/\sqrt{8}, 1/\sqrt{8}\right).$
The first interval corresponds to case (1), while the second is case (2). 

 We now analyze these cases separately taking into account (\ref{cov}).
In the first case, $T<-\frac{2}{\sqrt{27}}$ implies that $a_6<0$ and $-3 \left(a_6/2\right)^{2/3}< a_4<0.$ If $\text{ht}(E(a_4,a_6))\leq X$, then we have that $|a_4|\leq \sqrt[3]{X/4}$, and $|a_6|\leq \sqrt{X/27}.$ As $X$ approaches infinity, the proportion of such curves satisfying $-3 \left(a_6/2\right)^{2/3}< a_4<0,$ with $a_6<0,$ satisfies
\begin{eqnarray*}
\kappa_1 \ :=\ \lim_{X\rightarrow +\infty}\frac{3\displaystyle\int_{ -\sqrt{X/27} }^{0} \left(s/2\right)^{2/3}ds}{4\cdot \sqrt[3]{X/4}\cdot \sqrt{X/27}} 
\ = \ \frac{ 3}{20}.
\end{eqnarray*}
 For (2), we note that $|T|<\frac{1}{\sqrt{8}}$ implies $|a_6|< \frac{1}{\sqrt{8}} |a_4|^{3/2},$ and  so   
\begin{eqnarray*}
\kappa_2 \ :=\ \lim_{X\rightarrow +\infty} \frac{2 \displaystyle\int_{0}^{ \sqrt[3]{X/4}}\tfrac{1}{\sqrt{8}} s^{3/2}ds}{4\cdot \sqrt[3]{X/4}\cdot \sqrt{X/27}} 
\ = \ 
\frac{ 3 \sqrt{6}}{40}. 
\end{eqnarray*}
Therefore, Lemma~\ref{Nmin} shows that (\ref{formula}) is
$$
{\color{black}
\kappa\cdot P_{\Tam}(1)=\rho\cdot (\kappa_1+\kappa_2)\cdot P_{\Tam}(1)= 
\frac{21342914775}{228811\pi^{10}}\cdot \left(\frac{3}{20}+\frac{3\sqrt{6}}{40}\right)\cdot P_{\Tam}(1)
= 0.1679\dots.}
$$

\end{proof}

\medskip

\section{Appendix}

\setlength{\tabcolsep}{6pt} 
\renewcommand{\arraystretch}{2.}
\begin{center}
\begin{table}[!ht]
\begin{tabular}{|c|c|c|||c|c|c|||c|c|c|}
 \hline
Type & $c_p$ & $\delta_p'(K,\tp)$ & Type & $c_p$ & $\delta_p'(K,\tp)$ &Type &$c_p$ &$\delta_p'(K,\tp)$ \\ \hline \hline \hline
$I_0$ & $1$ & $\frac{p-1}{p}$ & $I_0^*$ & $1$ & $\frac{1}{3}\frac{(p^2-1)}{p^7}$ &$III$ & $2$ & $\frac{(p-1)}{p^4}$\\ \hline
$I_1$ & $1$ & $\frac{(p-1)^2}{p^3}$ & $I_0^*$ & $2$ & $\frac{1}{2}\frac{(p-1)}{p^6}$ &  $III^*$ & $2$ & $\frac{(p-1)}{p^9}$\\ \hline
$I_2$ & $2$ & $\frac{(p-1)^2}{p^4}$ & $I_0^*$ & $4$ & $\frac{1}{6}\frac{(p-1)(p-2)}{p^7}$  & $IV$ & $1$ & $\frac{1}{2}\frac{(p-1)}{p^5}$\\ \hline
$I_{n\geq3}$ & $\varepsilon(n)$ & $\frac{1}{2}\frac{(p-1)^2}{p^{n+2}}$ & $I^*_{n\geq 1}$ & $2$ & {\color{black}$\frac{1}{2}\frac{(p-1)^2}{p^{7+n}}$} & $IV$ & $3$ & $\frac{1}{2}\frac{(p-1)}{p^5}$\\ \hline
$I_{n\geq3}$ & $n$ & $\frac{1}{2}\frac{(p-1)^2}{p^{n+2}}$ & $I^*_{n\geq 1}$ & $4$ & {\color{black}$\frac{1}{2}\frac{(p-1)^2}{p^{7+n}}$} & $IV^*$ & $1$ & $\frac{1}{2}\frac{(p-1)}{p^8}$\\ \hline
$II$ & $1$ & $\frac{(p-1)}{p^3}$ & $II^*$ & $1$ & $\frac{(p-1)}{p^{10}}$   & $IV^*$ & $3$ & $\frac{1}{2}\frac{(p-1)}{p^8}$  \\ \hline
\end{tabular}
\smallskip
\caption{The $\delta_p'(K,\tp)$ for $p\geq 5$ (Note. $\varepsilon(n):=((-1)^n+3)/2.$)}
\label{pgeq5}
\end{table}
\end{center}

\setlength{\tabcolsep}{6pt} 
\renewcommand{\arraystretch}{2.}
\begin{center}
\begin{table}[!ht]
\small

\begin{tabular}{|c|c|c|||c|c|c|}
 \hline
Type & $c_2$ & $\widehat \delta_2(K,\tp)$& Type & $c_2$ & $\widehat \delta_2(K,\tp)$  \\ \hline \hline \hline
$I_0$ & $1$ & $\frac{1}{512}$&$I_{n\geq3,\text{odd}}$ & $1$ & $\frac{1}{2^{n+11}}$\\ \hline
$I_1$ & $1$ & $\frac{1}{2^{11}}$&$I_{n\geq3,\text{even}}$ & $2$ & $\frac{1}{2^{n+11}}$\\ \hline
$I_2$ & $2$ & $\frac{1}{2^{12}}$&$I_{n\geq3}$ & $n$ & $\frac{1}{2^{n+11}}$ 
\\ \hline
\end{tabular}
\quad
\begin{tabular}{|c|c|c|||c|c|c|}
 \hline
Type & $c_3$ & $\widehat \delta_3(K,\tp)$& Type & $c_3$ & $\widehat \delta_3(K,\tp)$  \\ \hline \hline \hline
$I_0$ & $1$ & $\frac{4}{3^{11}}$&$I_{n\geq3, {\text {\rm odd}}}$ & $1$ & $\frac{2}{3^{n+11}}$\\ \hline
$I_1$ & $1$ & $\frac{4}{3^{12}}$&$I_{n\geq3, {\text {\rm even}}}$ & $2$ & $\frac{2}{3^{n+11}}$\\ \hline
$I_2$ & $2$ & $\frac{4}{3^{13}}$&$I_{n\geq3}$ & $n$ & $\frac{2}{3^{n+11}}$ \\ \hline
\end{tabular}
\medskip
\caption{The $\widehat \delta_2(K,\tp)$ and $\widehat \delta_3(K,\tp)$ }
\label{phattwothree}
\end{table}
\normalsize
\end{center}

\setlength{\tabcolsep}{6pt} 
\renewcommand{\arraystretch}{2.}
\begin{center}
\begin{table}[!ht]
\small
{\color{black}
\begin{tabular}{|c|c|c|||c|c|c|}
 \hline
Type & $c_2$ & $\delta_2'(K,\tp)$ & Type & $c_2$ & $\delta_2'(K,\tp)$ \\ \hline \hline \hline
$I_0$ & $1$ & $0$ & $I_0^*$ & $1$ & $\frac{1}{32}$ \\ \hline
$I_1$ & $1$ & $0$ & $I_0^*$ & $2$ & $\frac{1}{32}$\\ \hline
$I_2$ & $2$ & $0$ & $I_0^*$ & $4$ & $0$ \\ \hline
$I_{n\geq3}$ & $\varepsilon(n)$ & $0$ & $I^*_{n\geq 1}$ & $2$ & $\frac{1}{2^{n+6}}$ \\ \hline
$I_{n\geq3}$ & $n$ & $0$ & $I^*_{n\geq 1}$ & $4$ & $\frac{1}{2^{n+6}}$ \\ \hline
$II$ & $1$ & $\frac{1}{2}$ & $II^*$ & $1$ & $\frac{1}{256}$  \\ \hline
$III$ & $2$ & $\frac{1}{4}$ &  $III^*$ & $2$ & $\frac{1}{128}$  \\ \hline
$IV$ & $1$ & $\frac{1}{16}$ &$IV^*$ & $1$ & $\frac{1}{128}$ \\ \hline
$IV$ & $3$ & $\frac{1}{16}$ & $IV^*$ & $3$ & $\frac{1}{128}$ \\ \hline
\end{tabular}
\quad
\begin{tabular}{|c|c|c|||c|c|c|}
 \hline
Type & $c_3$ & $\delta_3'(K,\tp)$ & Type & $c_3$ & $\delta_3'(K,\tp)$ \\ \hline \hline \hline
$I_0$ & $1$ & $\frac{2}{3}$ & $I_0^*$ & $1$ & $\frac{8}{3^7}$ \\ \hline
$I_1$ & $1$ & $0$ & $I_0^*$ & $2$ & $\frac{1}{3^5}$\\ \hline
$I_2$ & $2$ & $0$ & $I_0^*$ & $4$ & $\frac{1}{3^7}$ \\ \hline
$I_{n\geq3}$ & $\varepsilon(n)$ & $0$ & $I^*_{n\geq 1}$ & $2$ & $\frac{2}{3^{n+6}}$ \\ \hline
$I_{n\geq3}$ & $n$ & $0$ & $I^*_{n\geq 1}$ & $4$ & $\frac{2}{3^{n+6}}$ \\ \hline
$II$ & $1$ & $\frac{2}{9}$ & $II^*$ & $1$ & $\frac{2}{3^9}$  \\ \hline
$III$ & $2$ & $\frac{2}{27}$ &  $III^*$ & $2$ & $\frac{10}{3^9}$  \\ \hline
$IV$ & $1$ & $\frac{1}{81}$ &$IV^*$ & $1$ & $\frac{7}{3^9}$ \\ \hline
$IV$ & $3$ & $\frac{1}{81}$ & $IV^*$ & $3$ & $\frac{7}{3^9}$ \\ \hline
\end{tabular}
\smallskip
\caption{The $\delta_2'(K,\tp)$ and $\delta_3'(K,\tp)$ }
\label{ptwothree}}
\end{table}
\normalsize
\end{center}

\end{document}